\documentclass{article}

\usepackage{amsmath}
\usepackage{amssymb}
\usepackage{fullpage}
\usepackage{amsthm}
\usepackage{natbib}
\usepackage{amsfonts}
\usepackage{graphicx}
\usepackage{enumitem}
\usepackage{url}
\usepackage{graphicx}
\usepackage{tabularx}
\usepackage{parskip}
\newcommand{\diag}{\mathrm{diag}}
\newcommand{\vect}{\mathrm{vect}}

\newcommand{\VEC}{\mbox{vec}}
\newtheorem{thm}{Theorem}
\newtheorem{lmm}{Lemma}
\newtheorem{crly}{Corollary}

\newtheorem{algm}{Algorithm}
\newtheorem{dfn}{Definition}
\newcommand{\In}{{\rm{in}}}
\newcommand{\Out}{{\rm{out}}}
\newcommand{\tr}{{\rm{trace}}}
\newcommand{\bal}{{\rm{bal}}}
\newcommand{\pre}{{\rm{pre}}}
\newcommand{\post}{{\rm{post}}}
\newcommand{\per}{{\rm{per}}}
\newcommand{\eper}{{\rm{eper}}}
\newcommand{\new}{{\rm{new}}}

\usepackage{tikz}
\newcommand*\circled[1]{\tikz[baseline=(char.base)]{
            \node[shape=circle,draw,inner sep=2pt] (char) {#1};}}
\usepgflibrary{shapes}
\newcommand*\triangled[1]{\tikz[baseline=(char.base)]{
            \node[regular polygon, regular polygon sides=3,draw,inner sep=1pt] (char) {#1};}}
\usepackage{hyperref} 

\begin{document}

\title{Balanced Truncation Model Reduction of Nonstationary Systems Interconnected over Arbitrary Graphs}

\author{Dany Abou Jaoude and Mazen Farhood\\
\footnotesize Kevin T. Crofton Department of Aerospace and Ocean Engineering, Virginia Tech, Blacksburg, VA 24061, U.S.A.\\
\footnotesize Email addresses: danyabj@vt.edu, farhood@vt.edu.}

\date{}

\maketitle

\begin{abstract}                          
This paper deals with the balanced truncation model reduction of discrete-time, linear time-varying, heterogeneous subsystems interconnected over finite arbitrary directed graphs. The information transfer between the subsystems is subject to a communication latency of one time-step. The presented method guarantees the preservation of the interconnection structure and further allows for its simplification. In addition to truncating temporal states associated with the subsystems, the method allows for the order reduction of spatial states associated with the interconnections between the subsystems and even the removal of whole interconnections. Upper bounds on the $\ell_2$-induced norm of the resulting error system are derived. The proposed method is illustrated through an example.
\end{abstract}


\section{Introduction}
Various biological and engineering systems consist of multiple interacting agents. Mathematically describing such systems can lead to models with a very large number of states, especially in the case of ``large'' models of the agents and complicated interconnection structures. Thus arises the need for model order reduction to simplify the control analysis and synthesis problems. Interconnected systems can be treated as one global system; and standard model reduction tools, like balanced truncation (BT) and coprime factors reduction (CFR), can then be applied to these systems. However, such an approach is not always desirable as it does not guarantee the preservation of the interconnection structure of the system.

Various works have addressed the problem of structure-preserving BT and CFR for interconnected systems. See, for example, \begin{NoHyper}{\cite{li2005,murray2009,Altaie2015,DAJMFCDC15,aboujaoudefarhood2016CFR}}\end{NoHyper}. The methods in these references, like the one in \begin{NoHyper}{\cite{beck1996}}\end{NoHyper}, are based on the existence of block-diagonally structured solutions to linear matrix inequalities (LMIs). BT applies to stable systems, guarantees the stability of the reduced order system, and comes with a guaranteed upper bound on the norm of the error system. On the other hand, CFR applies to stabilizable and detectable systems, and guarantees the stabilizability and detectability of the reduced order system. However, the bound from CFR is not in terms of the norm of the error between the full and reduced order systems, but rather, it is in terms of the norm of the error between their corresponding coprime factorizations. The aforementioned works can be classified based on the modeling of the interconnections between the subsystems. Namely, \begin{NoHyper}{\cite{murray2009}}\end{NoHyper} account for the interconnection structure using a transfer function matrix, whereas, \begin{NoHyper}{\cite{Altaie2015,DAJMFCDC15,aboujaoudefarhood2016CFR}}\end{NoHyper} model the interconnections between the subsystems as states, which we refer to as spatial states. The latter methods allow for the order reduction of the spatial states in addition to the standard states associated with the subsystems, which we refer to as the temporal states. That is, in addition to guaranteeing the preservation of the interconnection structure, these methods  further allow for its simplification. In particular, \begin{NoHyper}{\cite{Altaie2015}}\end{NoHyper} deal with homogeneous, linear time-invariant (LTI) or linear parameter-varying (LPV) subsystems, interconnected over a grid. Model reduction is only applied to one subsystem, and then, all the temporal states, all the forward spatial states, and all the backward spatial states are truncated in a uniform way, respectively. As for  \begin{NoHyper}{\cite{DAJMFCDC15,aboujaoudefarhood2016CFR}}\end{NoHyper}, they deal with heterogeneous, linear time-varying (LTV) subsystems, interconnected over arbitrary directed graphs. The methods therein allow for individually truncating each of the temporal and spatial states, and even permit the removal of a whole interconnection if it is deemed negligible. Due to the time-varying nature of the subsystems, these methods usually involve solving infinite sequences of LMIs. Also, if truncation is performed at infinitely many time-steps, the resulting error bound might not be finite. \begin{NoHyper}{\cite{DAJMFCDC15}}\end{NoHyper} show that in the special case of time-periodic subsystems, the sequences of LMIs can be restricted to the first time-period, and the error bound is guaranteed to be finite.

The current work extends the results of \begin{NoHyper}{\cite{DAJMFCDC15}}\end{NoHyper} from time-periodic subsystems to eventually time-periodic subsystems, i.e., subsystems which become time-periodic after some initial finite time-horizon. We also derive a tighter expression for the error bound for general LTV subsystems, which applies when the entries of the balanced gramians corresponding to the truncated states form monotonic sequences in time. We also provide an illustrative example.

The paper is organized as follows. In Section~\ref{preliminariessection}, we define the notation and summarize the adopted state-space framework along with the relevant analysis results. Then, in Section~\ref{balancedtruncationmodelreductionsection}, we present the BT method. We derive the error bounds in Section~\ref{errorboundssection} and treat the special class of eventually time-periodic subsystems in Section~\ref{ETPsection}. In Section~\ref{examplesection}, we apply the method to an example. We conclude the paper with Section~\ref{conclusionsection}.

\section{Preliminaries}\label{preliminariessection}
\subsection{Notation}
\begin{figure}[t]
\centering
\includegraphics[scale=0.25]{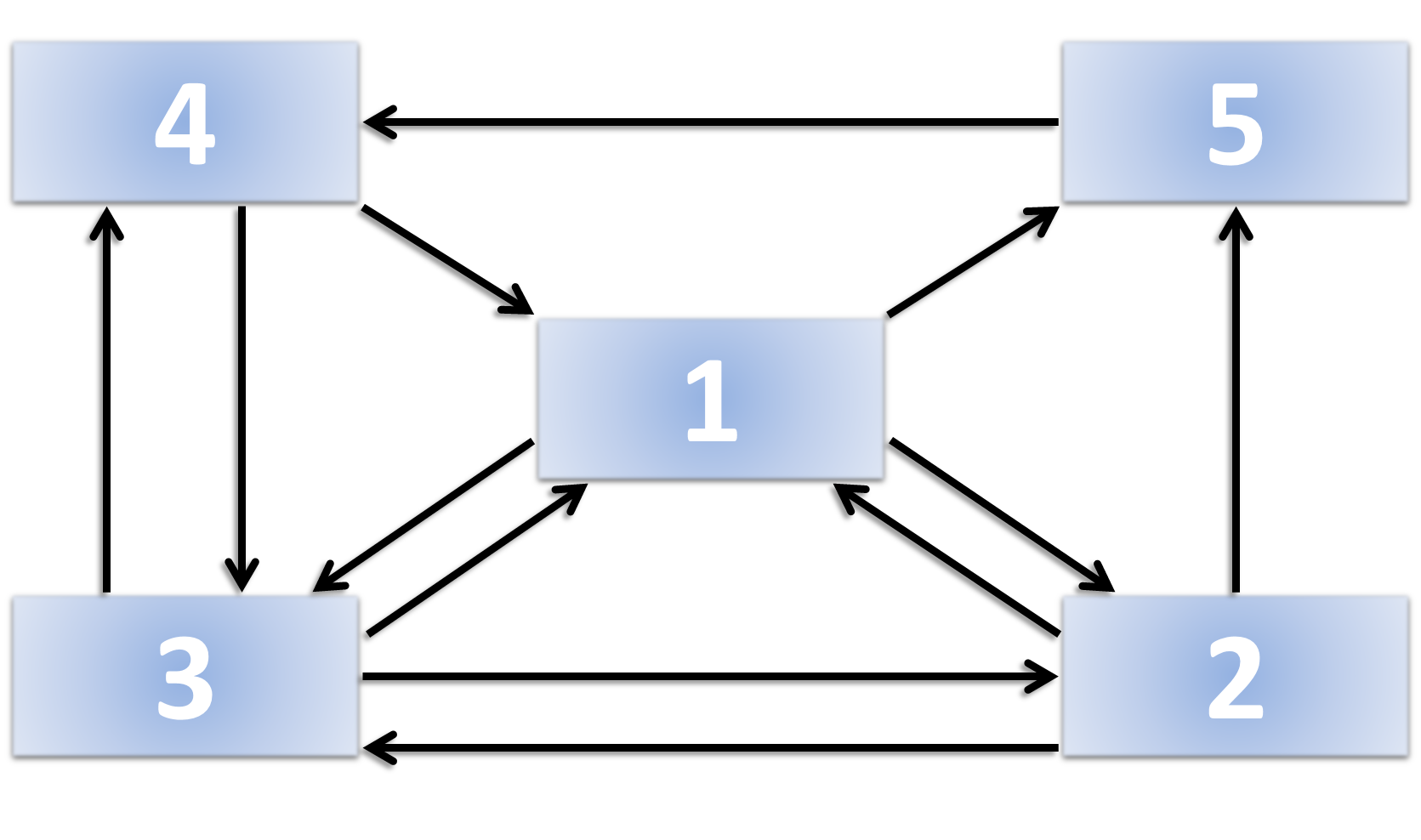}
\vskip -6mm
\caption{Example of a directed graph.}
\vskip -3mm
\label{fig:digraphillustration}
\end{figure}
The sets of nonnegative integers, real numbers, and $n \times n$ symmetric matrices are denoted by $\mathbb{N}_0$, $\mathbb{R}$, and $\mathbb{S}^n$, respectively. Let $S$ be an ordered subset of $\mathbb{N}_0$. We denote by $(v_i)_{i \in S}$ the vector-valued sequence associated with $S$ and by $\VEC(v_i)_{i \in S}$ the vertical concatenation of the elements of $(v_i)_{i \in S}$. The elements in $(v_i)_{i \in S}$ are ordered conformably with the elements in $S$. As an example, let $S=\{1,2,4\}$. Then, $(v_i)_{i \in S}= (v_1,v_2,v_4)$ and $\VEC(v_i)_{i \in S}=\begin{bmatrix}v_1^T&v_2^T&v_4^T\end{bmatrix}^T$. Similarly, we denote by $(M_i)_{i \in S}$ the matrix-valued sequence associated with $S$ and by $\diag(M_i)_{i \in S}$ the block-diagonal augmentation of the elements of $(M_i)_{i \in S}$. In our example, $(M_i)_{i \in S}=(M_1,M_2,M_4)$ and $\diag(M_i)_{i \in S}=\diag(M_1,M_2,M_4)=\begin{bmatrix}M_1&0&0\\0&M_2&0\\0&0&M_4\end{bmatrix}$. $0_{i\times j}$ denotes an $i\times j$ zero matrix and $I_i$ denotes an $i\times i$ identity matrix.

$\mathcal{G}(V,E)$ refers to a directed graph with set of vertices $V$ and set of directed edges $E$. We assume throughout that the directed graph under consideration is finite. That is, both $V$ and $E$ are finite sets. We denote by $N$ the finite number of vertices, and choose $V=\{1,\ldots,N\}$.  The ordered pair $(i, j)$ is in $E$ if there exists a directed edge from $i \in V$ to $j \in V$.
For each $k \in V$, we define the set of vertices with an outgoing edge to $k$ as $E_{\In}^{(k)}=\{i\in V \,|\, (i,k) \in E\}$, and denote its cardinality by $m(k)$. Similarly, we define the set of vertices with an incoming edge from $k$ as $E_{\Out}^{(k)}=\{j\in V \,|\, (k,j) \in E\}$, and denote its cardinality by $p(k)$. We order the elements in these sets in an increasing fashion. For example, consider the directed graph in Figure~\ref{fig:digraphillustration}. For $k=1$, we define $E_{\In}^{(1)}=\{2,3,4\}$, $m(1)=3$, $E_{\Out}^{(1)}=\{2,3,5\}$, $p(1)=3$, and so on.

$X \prec 0$ (resp. $\succ 0$) means that the symmetric matrix $X$ is negative definite (resp. positive definite). Let $n$ be an integer sequence such that $n: (t,k)\in\mathbb{N}_0\times V$ $\rightarrow n^{(k)}(t)\in\mathbb{N}_0$. We define $\ell(\{\mathbb{R}^{n^{(k)}(t)}\})$ as the vector space of mappings $w: (t,k) \in \mathbb{N}_0 \times V \rightarrow w^{(k)}(t) \in \mathbb{R}^{n^{(k)}(t)}$. The Hilbert space $\ell_2(\{\mathbb{R}^{n^{(k)}(t)}\})$ is the subspace of $\ell(\{\mathbb{R}^{n^{(k)}(t)}\})$ consisting of mappings $w$ with a finite $\ell_2$-norm $\|w\|_2=\left(\sum_{(t,k)}w^{(k)}(t)^T w^{(k)}(t)\right)^\frac{1}{2}$. We subsequently use the abbreviated symbols $\ell$ and $\ell_2$.

\subsection{State-Space Representation}
Next, we give the state-space equations for a distributed system $G$, formed by discrete-time, heterogeneous, LTV subsystems interconnected over finite arbitrary directed graphs and subjected to a communication latency. We represent the interconnection structure of $G$ using a directed graph $\mathcal{G}(V,E)$, where each subsystem $G^{(k)}$ corresponds to a vertex $k\in V$, and the interconnections between the subsystems are described by the directed edges. Each subsystem $G^{(k)}$ has a discrete-time LTV model, with states $x^{(k)}(t)$, inputs $u^{(k)}(t)$, and outputs $y^{(k)}(t)$. We refer to the states associated with the subsystems as temporal states or  vertex/node states. The interconnections between the subsystems are also modeled using states, which we refer to as spatial states or  interconnection/edge states. Namely, we associate a state $x^{(ij)}(t)$ with each edge $(i,j)\in E$. Due to the communication latency, the information sent from $G^{(i)}$ at time $t$ reaches $G^{(j)}$ at the next time-step, i.e., $t+1$. For each subsystem $G^{(k)}$, we define vectors $x_{\In}^{(k)}(t)=\VEC(x^{(ik)}(t))_{i\in E_{\In}^{(k)}}$ and $x_{\Out}^{(k)}(t)=\VEC(x^{(kj)}(t))_{j\in E_{\Out}^{(k)}}$, which are partitioned into $m(k)$ and $p(k)$ vector-valued channels, respectively. These vectors represent the total information received and sent by $G^{(k)}$ at time $t$. When all subsystems are considered, and since the interconnection input to a subsystem is an output to another subsystem, both $x^{(k)}_{\Out}$ and $x^{(k)}_{\In}$ contain all spatial states $x^{(ij)}$. We assume zero initial conditions for the temporal and the spatial states. Then, for all $(t,k) \in \mathbb{N}_0\times V$, we have  $x^{(k)}(0)=0$,  $x_{\In}^{(k)}(0)=0$,
\begin{align}
\begin{bmatrix}
x^{(k)}(t+1)\\
x_{\Out}^{(k)}(t+1)
\end{bmatrix}&=A^{(k)}(t) \begin{bmatrix}
x^{(k)}(t)\\
x_{\In}^{(k)}(t)
\end{bmatrix} +B^{(k)}(t)\, u^{(k)}(t),\nonumber \\
y^{(k)}(t)&=C^{(k)}(t)\begin{bmatrix}
x^{(k)}(t)\\
x_{\In}^{(k)}(t)
\end{bmatrix} +D^{(k)}(t)\, u^{(k)}(t). \label{eq:distributedsystemequations}
\end{align}
The matrix-valued sequences of state-space matrices, e.g., $A^{(k)}(t)$, are known a priori and assumed to be uniformly bounded. The dimensions of signals $x^{(k)}(t)$, $u^{(k)}(t)$, $y^{(k)}(t)$, and $x^{(ij)}(t)$ can vary with $t$, and $k$ or $(i,j)$, and are denoted by $n^{(k)}(t), n_u^{(k)}(t)$, $n_y^{(k)}(t)$, and $n^{(ij)}(t)$, respectively, for all $(t,k)\in \mathbb{N}_0\times V$ and $(i,j) \in E$. We denote the realization of system $G$ by the quadruple $\left(A^{(k)}(t), B^{(k)}(t), C^{(k)}(t), D^{(k)}(t)\right)$.

For each $(t,k)$, the state-space matrices are naturally partitioned conformably with the partitioning of $\begin{bmatrix}x^{(k)}(t+1)^T& x_{\Out}^{(k)}(t+1)^T\end{bmatrix}^T$  and $\begin{bmatrix}x^{(k)}(t)^T& x_{\In}^{(k)}(t)^T\end{bmatrix}^T$. For example, consider the distributed system in Figure~\ref{fig:digraphillustration}. The state-space matrices of subsystem $G^{(1)}$ are partitioned as follows:
\[A^{(1)}(t){=}\!\begin{bmatrix}
\rule{0mm}{4.5mm}A_{00}^{(1)}(t) & \!\!A_{02}^{(1)}(t)& \!\!A_{03}^{(1)}(t) & \!\!A_{04}^{(1)}(t)\\
\rule{0mm}{4.5mm}A_{20}^{(1)}(t) & \!\!A_{22}^{(1)}(t)& \!\!A_{23}^{(1)}(t) & \!\!A_{24}^{(1)}(t)\\
\rule{0mm}{4.5mm}A_{30}^{(1)}(t) & \!\!A_{32}^{(1)}(t)& \!\!A_{33}^{(1)}(t) & \!\!A_{34}^{(1)}(t)\\
\rule{0mm}{4.5mm}A_{50}^{(1)}(t) & \!\!A_{52}^{(1)}(t)& \!\!A_{53}^{(1)}(t) & \!\!A_{54}^{(1)}(t)
\end{bmatrix}\!\!{,} \,\,
B^{(1)}(t){=}\!\begin{bmatrix}
\rule{0mm}{4.5mm}B_0^{(1)}(t)\\
\rule{0mm}{4.5mm}B_2^{(1)}(t)\\
\rule{0mm}{4.5mm}B_3^{(1)}(t)\\
\rule{0mm}{4.5mm}B_5^{(1)}(t)
\end{bmatrix}\!\!{,}\,\,C^{(1)}(t){=}\!\begin{bmatrix}
C_0^{(1)}(t) & \!\!C_2^{(1)}(t) & \!\!C_3^{(1)}(t) & \!\!C_4^{(1)}(t)
\end{bmatrix}\!\!{,} \]
where $A_{00}^{(1)}(t)$ is an $n^{(1)}(t+1) \times n^{(1)}(t)$ matrix, $A_{20}^{(1)}(t)$ is an $n^{(12)}(t+1) \times n^{(1)}(t)$ matrix, $B_0^{(1)}(t)$ is an $n^{(1)}(t+1) \times n_u^{(1)}(t)$ matrix, $C_0^{(1)}(t)$ is an $n_y^{(1)}(t) \times n^{(1)}(t)$ matrix,  $C_2^{(1)}(t)$ is an $n_y^{(1)}(t) \times n^{(21)}(t)$ matrix, etc.

\subsection{Analysis Results}
We now summarize the relevant analysis results of \begin{NoHyper}{\cite{farhood2015}}\end{NoHyper}. We consider a system $G$ with realization $(A^{(k)}(t), B^{(k)}(t), C^{(k)}(t), D^{(k)}(t))$. Since the state-space equations (\ref{eq:distributedsystemequations}) have zero initial conditions and the state-space matrices are defined for $t \in \mathbb{N}_0$, we can equivalently assume that the state-space matrices are zeros for $t<0$. Then, from \begin{NoHyper}{\cite{dullerud2004}}\end{NoHyper}, it can be shown that system $G$ is well-posed, i.e., given inputs in $\ell$, the state-space equations admit unique solutions in $\ell$, and further define a linear causal mapping on $\ell$. A well-posed system $G$ is said to be stable if, given inputs in $\ell_2$, the state-space equations admit unique solutions in $\ell_2$, and further define a linear causal mapping on $\ell_2$. Next, we give a Lyapunov-based test to check if system $G$ is stable. The result constitutes the basis of the BT method.
\vskip 3mm
\begin{lmm}\label{strongstabilitylemma}
System $G$ is stable if there exist $\beta>0$ and uniformly bounded, positive definite, matrix-valued functions $X^{(k)}(t)\in \mathbb{S}^{n^{(k)}(t)}$ and $X^{(ij)}(t)\in \mathbb{S}^{n^{(ij)}(t)}$, for all $(t,k)\in \mathbb{N}_0\times V$ and $(i,j) \in E$, such that $X^{(k)}(t) \succ \beta I$, $X^{(ij)}(t)\succ \beta I$, and
\begin{equation}\label{eq:strongstabilitycondition}
A^{(k)}(t)^T
\begin{bmatrix}
X^{(k)}(t+1)& 0\\
0 & X_{\Out}^{(k)}(t+1)
\end{bmatrix}A^{(k)}(t)
-\begin{bmatrix}
X^{(k)}(t)&0\\
0& X_{\In}^{(k)}(t)
\end{bmatrix} \prec -\beta I,
\end{equation}
\[
X_{\In}^{(k)}(t) =\diag(X^{(ik)}(t))_{i\in E_{\In}^{(k)}}, \quad \quad \mbox{and} \quad \quad
X_{\Out}^{(k)}(t)=\diag(X^{(kj)}(t))_{j\in E_{\Out}^{(k)}}.\]
\end{lmm}

The solutions to (\ref{eq:strongstabilitycondition}) can be classified into temporal terms  $X^{(k)}(t)$ and spatial terms $X^{(ij)}(t)$. Due to the time-varying nature of the subsystems, there is an infinite sequence of LMIs associated with each subsystem. Moreover, the LMI sequence associated with a given subsystem is coupled with the LMI sequences of the other subsystems through the spatial terms. The ``$\beta I$'' terms in (\ref{eq:strongstabilitycondition})  are small quantities added to ensure that the matrix sequences on the left-hand side do not converge to singular matrices as $t$ approaches infinity. Note that the conditions $X^{(k)}(t) \succ \beta I $ and $X^{(ij)}(t) \succ \beta I$ are implied by (\ref{eq:strongstabilitycondition}). Even though redundant, these conditions are explicitly given in the statement of the lemma to stress that the terms  $X^{(k)}(t)$ and $X^{(ij)}(t)$ do not approach singular matrices as $t$ approaches infinity, as explained above. Uniform positive definiteness will be used in due course to guarantee the invertibility and the boundedness of the inverses of such terms. Subsequently, we no longer specify the dimensions of $X^{(k)}(t)$ and $X^{(ij)}(t)$.

Systems that satisfy the conditions in Lemma \ref{strongstabilitylemma} are called strongly stable. Lemma~\ref{strongstabilitylemma} provides a sufficient condition for stability, and so strong stability implies stability, but the converse is not always true. Specifically, strongly stable systems are stable systems which have the required structured solutions to (\ref{eq:strongstabilitycondition}). The proposed BT scheme suffers from conservatism as it only applies to strongly stable systems. However, this imposed structure on the solutions of (\ref{eq:strongstabilitycondition}) allows for the preservation and the simplification of the interconnection structure during the model reduction process. \begin{NoHyper}{\cite{Trnka2013,sootla2016}}\end{NoHyper} identify classes of systems with guaranteed structured solutions to LMIs.

For a stable system $G$ mapping $u \in \ell_2$ to $y\in \ell_2$ and starting from zero initial conditions, the $\ell_2$-induced norm is defined as $\|G\|=\sup_{0 \neq u\in\ell_2} \frac{\|y\|_2}{\|u\|_2}$.
\vskip 3mm
\begin{lmm}\label{performancelemma}
System $G$ is strongly stable and satisfies $\|G\|< \gamma$, for some $\gamma>0$, if there exist $\beta>0$ and uniformly bounded, positive definite, matrix-valued functions $X^{(k)}(t)\succ \beta I$ and $X^{(ij)}(t)\succ \beta I$, for all $(t,k)\in \mathbb{N}_0\times V$ and $(i,j) \in E$, such that
\begin{equation}\label{eq:performanceinequality}
\begin{bmatrix}A^{(k)}(t) & B^{(k)}(t)\\ C^{(k)}(t) & D^{(k)}(t)\end{bmatrix}^T
\begin{bmatrix}
X^{(k)}(t+1)& 0 & 0\\
0 & X_{\Out}^{(k)}(t+1) & 0\\
0 & 0 & I
\end{bmatrix}\begin{bmatrix}A^{(k)}(t) & B^{(k)}(t)\\ C^{(k)}(t) & D^{(k)}(t)\end{bmatrix}
-\begin{bmatrix}
X^{(k)}(t) & 0 & 0\\
0  & X_{\In}^{(k)}(t) & 0\\
0  & 0 & \gamma^2 I
\end{bmatrix} \prec -\beta I.
\end{equation}
\end{lmm}

\section{Balanced Truncation Model Reduction}\label{balancedtruncationmodelreductionsection}
\subsection{Balanced Realization}
We now extend the notions of generalized Lyapunov inequalities and generalized gramians, discussed in \begin{NoHyper}{\cite{hinrichsen1990,beck1996}}\end{NoHyper}, to the class of distributed systems. Namely, the controllability and observability generalized gramians are uniformly bounded, positive definite, matrix-valued functions, denoted by $X^{(k)}(t)$, $X^{(ij)}(t)$ and $Y^{(k)}(t)$, $Y^{(ij)}(t)$, respectively, which, for some scalar $\beta > 0$, satisfy $X^{(k)}(t)\succ \beta I$, $X^{(ij)}(t)\succ \beta I$ and $Y^{(k)}(t)\succ \beta I$, $Y^{(ij)}(t)\succ \beta I$, in addition to the following generalized Lyapunov inequalities:
\begin{align}
A^{(k)}(t)
\begin{bmatrix}
X^{(k)}(t) & 0 \\
0 & X_{\In}^{(k)}(t)
\end{bmatrix}A^{(k)}(t)^T
-\begin{bmatrix}
X^{(k)}(t+1)&0\\
0& X_{\Out}^{(k)}(t+1)\end{bmatrix} +B^{(k)}(t)B^{(k)}(t)^T \prec -\beta I, \label{eq:lyapinequ1}\\
A^{(k)}(t)^T
\begin{bmatrix}
Y^{(k)}(t+1)& 0\\
0 & Y_{\Out}^{(k)}(t+1)\end{bmatrix}A^{(k)}(t)-
\begin{bmatrix}
Y^{(k)}(t)&0\\
0& Y_{\In}^{(k)}(t)
\end{bmatrix}+C^{(k)}(t)^T C^{(k)}(t)\prec -\beta I, \label{eq:lyapinequ2}
\end{align}
where $Y_{\In}^{(k)}(t)$ and $Y_{\Out}^{(k)}(t)$ are defined similarly to $X_{\In}^{(k)}(t)$ and $X_{\Out}^{(k)}(t)$. The generalized Lyapunov inequalities and the generalized gramians allow for the definition of a balanced realization of a distributed system as given next.
\vskip 3mm
\begin{dfn}\label{balancedrealizationdefinition}
The realization of system $G$ is said to be balanced if there exist $\beta>0$ and uniformly bounded, diagonal, positive definite, matrix-valued functions $\Sigma^{(k)}(t)\succ \beta I$ and $\Sigma^{(ij)}(t)\succ \beta I$, for all $(t,k)\in \mathbb{N}_0\times V$ and $(i,j) \in E$, that simultaneously satisfy (\ref{eq:lyapinequ1}) and (\ref{eq:lyapinequ2}), i.e., $\Sigma^{(k)}(t)=X^{(k)}(t)=Y^{(k)}(t)$, $\Sigma^{(ij)}(t)=X^{(ij)}(t)=Y^{(ij)}(t)$, and $\Sigma^{(k)}(t)$ and $\Sigma^{(ij)}(t)$ are diagonal matrices.
\end{dfn}

We can show (see the proof of  Theorem \ref{eventuallyperiodicsolutionstostability}) that the existence of solutions to (\ref{eq:strongstabilitycondition}) is equivalent to the existence of solutions to (\ref{eq:lyapinequ1}) and (\ref{eq:lyapinequ2}), respectively. Thus, the generalized gramians are only defined for strongly stable systems. These gramians can be used to construct a balanced realization and balanced generalized gramians $\Sigma^{(k)}(t)$, $\Sigma^{(ij)}(t)$ for a given system $G$, as outlined next.
\vskip 3mm
\begin{algm}\label{algorithm}
We construct a balanced realization for a given strongly stable system $G$ with generalized gramians $X^{(k)}(t)$, $X^{(ij)}(t)$ and $Y^{(k)}(t)$, $Y^{(ij)}(t)$ as follows.
First, we compute the Cholesky factorizations
\[X^{(k)}(t)=R^{(k)}(t)^TR^{(k)}(t)\quad \quad \mbox{and} \quad \quad  Y^{(k)}(t)=H^{(k)}(t)^TH^{(k)}(t),\]
and we perform the singular value decompositions
\[H^{(k)}(t)R^{(k)}(t)^T=U^{(k)}(t)\Sigma^{(k)}(t)V^{(k)}(t)^T.\]
Then, we define the balancing transformations
\[T^{(k)}(t)=\Sigma^{(k)}(t)^{-1/2}\,\,U^{(k)}(t)^T\,\,H^{(k)}(t) \quad \quad \mbox{and} \quad \quad
T^{(k)}(t)^{-1}=R^{(k)}(t)^T\,\,V^{(k)}(t)\,\,\Sigma^{(k)}(t)^{-1/2}.\]
Similar steps are repeated for the spatial terms. $\Sigma^{(k)}(t)$, $\Sigma^{(ij)}(t)$ are the balanced generalized gramians.
We augment the obtained transformations as in
\[T_{\pre}^{(k)}(t)=\diag(T^{(k)}(t),T_{\Out}^{(k)}(t)), \quad \quad \mbox{and} \quad \quad T_{\post}^{(k)}(t)=\diag(T^{(k)}(t)^{-1},T_{\In}^{(k)}(t)^{-1}).\]
A balanced realization  $(A_{\bal}^{(k)}(t), B_{\bal}^{(k)}(t), C_{\bal}^{(k)}(t), D^{(k)}(t))$ of $G$ is then given by \[A_{\bal}^{(k)}(t)=T_{\pre}^{(k)}(t+1)A^{(k)}(t)T_{\post}^{(k)}(t),\quad B_{\bal}^{(k)}(t)=T_{\pre}^{(k)}(t+1)B^{(k)}(t), \quad \mbox{and} \quad C_{\bal}^{(k)}(t)=C^{(k)}(t)T_{\post}^{(k)}(t).\]
\end{algm}

An alternative algorithm is also given in \begin{NoHyper}{\cite{DAJMFCDC15}}\end{NoHyper}. Clearly, the balanced realization of a strongly stable system is not unique as it depends on the followed algorithm as well as the solutions to (\ref{eq:lyapinequ1}) and (\ref{eq:lyapinequ2}) used in the algorithm. To obtain useful results for model reduction, namely entries in the balanced generalized gramians that yield reasonable error bounds, we use a trace heuristic, i.e., we find generalized gramians with minimum sum of traces, see e.g., \begin{NoHyper}{\cite{farhood2007}}\end{NoHyper} and \begin{NoHyper}{\cite{beck2014}}\end{NoHyper}. An alternative heuristic is proposed in \begin{NoHyper}{\cite{Altaie2015}}\end{NoHyper}.

\subsection{Balanced Truncation}
Let $G$ be a distributed system with a balanced realization $\left(A^{(k)}(t), B^{(k)}(t), C^{(k)}(t), D^{(k)}(t)\right)$. We assume, without loss of generality, that the diagonal entries of the balanced generalized gramians are ordered in a decreasing fashion. The essence of BT is to truncate the state variables associated with the negligible entries. We partition the gramians into two blocks: one corresponding to the non-truncated states and the other to the truncated states. We illustrate the partitioning process for the temporal terms. Given integers $r^{(k)}(t)$, such that $0\le r^{(k)}(t) \le n^{(k)}(t)$, we partition $\Sigma^{(k)}(t)$  as
$\Sigma^{(k)}(t)=\diag(\Gamma^{(k)}(t), \,\Omega^{(k)}(t))$,
where $\Gamma^{(k)}(t) \in \mathbb{S}^{r^{(k)}(t)}$ are reduced order gramians associated with the non-truncated states and $\Omega^{(k)}(t)$ correspond to the truncated states. By allowing $r^{(k)}(t)$ to be equal to $0$ or $n^{(k)}(t)$ for some $(t,k)$, we allow that either all or no variables be truncated from the corresponding temporal state $x^{(k)}(t)$. This results in either $\Gamma^{(k)}(t)$ or $\Omega^{(k)}(t)$ having a zero dimension, which is a slight abuse of notation. The proposed method allows for the evaluation of the importance of a particular interconnection, and accordingly, the reduction of the dimension of the spatial state vector associated with it and even the complete removal of the interconnection. For example, if $r^{(ij)}(t)=0$, for all $t\in\mathbb{N}_0$, then the edge $(i,j)$ is removed altogether from the interconnection structure.

The next step is to partition the blocks of the state-space matrices in accordance with the partitioning of the blocks of $\diag(\Sigma^{(k)}(t+1),\Sigma_{\Out}^{(k)}(t+1))$ and $\diag(\Sigma^{(k)}(t),\Sigma_{\In}^{(k)}(t))$. Consider subsystem $G^{(1)}$ in Figure \ref{fig:digraphillustration}. $A_{00}^{(1)}(t)$ is partitioned according to the partitioning of $\Sigma^{(1)}(t+1)=\diag(\Gamma^{(1)}(t+1), \,\Omega^{(1)}(t+1))$  and $\Sigma^{(1)}(t)=\diag(\Gamma^{(1)}(t), \,\Omega^{(1)}(t))$ as in
\[A_{00}^{(1)}(t)=\begin{bmatrix}
\hat{A}_{00}^{(1)}(t)& A_{00_{12}}^{(1)}(t)\\
 A_{00_{21}}^{(1)}(t)& A_{00_{22}}^{(1)}(t)
\end{bmatrix}, \mbox{ where $\hat{A}_{00}^{(1)}(t)$ is an $r^{(1)}(t+1){\times} r^{(1)}(t)$ matrix.}\]
$B_0^{(1)}(t)$ is partitioned conformably with the partitioning of $\Sigma^{(1)}(t+1)=\diag(\Gamma^{(1)}(t+1), \,\Omega^{(1)}(t+1))$ as in \[B_0^{(1)}(t)=\begin{bmatrix}
\hat{B}_0^{(1)}(t)\\
  B_{0_2}^{(1)}(t)\end{bmatrix}, \mbox{ where $\hat{B}_0^{(1)}(t)$ is an $r^{(1)}(t+1){\times} n_u^{(1)}(t)$ matrix.}\]
Likewise, $C_0^{(1)}(t)$ is partitioned according to the partitioning of $\Sigma^{(1)}(t)=\diag(\Gamma^{(1)}(t), \,\Omega^{(1)}(t))$, namely,
\[C_0^{(1)}(t)=\begin{bmatrix}
\hat{C}_0^{(1)}(t)&
 C_{0_2}^{(1)}(t)
\end{bmatrix}, \mbox{ where $\hat{C}_0^{(1)}(t)$ is an $n_y^{(1)}(t)\times r^{(1)}(t)$ matrix, and so on.}\]
Now, we form the realization $(A_r^{(k)}(t), B_r^{(k)}(t), C_r^{(k)}(t), D^{(k)}(t))$ of the reduced order system $G_r$. For $A_r^{(k)}(t)$, $B_r^{(k)}(t)$, and $C_r^{(k)}(t)$, we keep the blocks that correspond to the non-truncated states, i.e., the partitions marked with a hat, e.g., $C_r^{(1)}(t)=\begin{bmatrix}
\hat{C}_0^{(1)}(t) & \hat{C}_2^{(1)}(t)& \hat{C}_3^{(1)}(t) & \hat{C}_4^{(1)}(t)\end{bmatrix}$,
\[A_r^{(1)}(t)=\begin{bmatrix}
\rule{0mm}{5mm}\hat{A}_{00}^{(1)}(t) & \hat{A}_{02}^{(1)}(t)& \hat{A}_{03}^{(1)}(t) & \hat{A}_{04}^{(1)}(t)\\
\rule{0mm}{5mm}\hat{A}_{20}^{(1)}(t) & \hat{A}_{22}^{(1)}(t)& \hat{A}_{23}^{(1)}(t) & \hat{A}_{24}^{(1)}(t)\\
\rule{0mm}{5mm}\hat{A}_{30}^{(1)}(t) & \hat{A}_{32}^{(1)}(t)& \hat{A}_{33}^{(1)}(t) & \hat{A}_{34}^{(1)}(t)\\
\rule{0mm}{5mm}\hat{A}_{50}^{(1)}(t) & \hat{A}_{52}^{(1)}(t)& \hat{A}_{53}^{(1)}(t) & \hat{A}_{54}^{(1)}(t)
\end{bmatrix}, \quad \quad \mbox{and} \quad \quad
B_r^{(1)}(t)=\begin{bmatrix}
\hat{B}_0^{(1)}(t)\\
\rule{0mm}{5mm}\hat{B}_2^{(1)}(t)\\
\rule{0mm}{5mm}\hat{B}_3^{(1)}(t)\\
\rule{0mm}{5mm}\hat{B}_5^{(1)}(t)
\end{bmatrix}.\]
It will be useful to permute the original state-space matrices and balanced gramians in order to group together the non-truncated blocks. For example,
\[
A_{b}^{(k)}(t)=\begin{bmatrix}
\rule{0mm}{4.5mm}A_r^{(k)}(t) & \bar{A}_{12}^{(k)}(t)\\
\rule{0mm}{4.5mm}\bar{A}_{21}^{(k)}(t) & \bar{A}_{22}^{(k)}(t)
\end{bmatrix}, \quad \quad \Gamma^{\In}_k(t)=
\begin{bmatrix}
\rule{0mm}{4.5mm}\Gamma^{(k)}(t) & 0\\
\rule{0mm}{4.5mm}0 & \Gamma_{\In}^{(k)}(t)
\end{bmatrix},\quad \,\, \mbox{and} \,\, \quad \Gamma^{\Out}_k(t)=
\begin{bmatrix}
\rule{0mm}{4.5mm}\Gamma^{(k)}(t)&0\\
\rule{0mm}{4.5mm}0& \Gamma_{\Out}^{(k)}(t)\end{bmatrix},\]
for appropriately defined $\bar{A}_{12}^{(k)}(t)$, $\bar{A}_{21}^{(k)}(t)$, and $\bar{A}_{22}^{(k)}(t)$. We define $B_{b}^{(k)}(t)$ and $C_b^{(k)}(t)$ similarly to $A_{b}^{(k)}(t)$, and $\Omega^{\In}_k(t)$ and $\Omega^{\Out}_k(t)$ similarly to $\Gamma^{\In}_k(t)$ and $\Gamma^{\Out}_k(t)$, respectively.
\vskip 3mm
\begin{lmm}\label{reducedrealizationstableandbalanced}
System $G_r$ is strongly stable and the given realization $(A_r^{(k)}(t), B_r^{(k)}(t), C_r^{(k)}(t), D^{(k)}(t))$ is balanced with balanced generalized gramians $\Gamma^{(k)}(t)$ and $\Gamma^{(ij)}(t)$, for all $(t,k)\in \mathbb{N}_0\times V$ and $(i,j) \in E$.
\end{lmm}
\begin{proof}
Since the realization of $G$ is balanced, then there exist balanced generalized gramians simultaneously solving (\ref{eq:lyapinequ1}) and (\ref{eq:lyapinequ2}). Applying appropriate permutations to (\ref{eq:lyapinequ1}) and (\ref{eq:lyapinequ2}), we obtain
\begin{align}
A_{b}^{(k)}(t)
\begin{bmatrix}
\Gamma^{\In}_k(t)&0\\
0&\Omega^{\In}_k(t)
\end{bmatrix}A_{b}^{(k)}(t)^T
-\begin{bmatrix}
\Gamma^{\Out}_k(t+1)&0\\
0&\Omega^{\Out}_k(t+1)
\end{bmatrix}+B_{b}^{(k)}(t)B_{b}^{(k)}(t)^T\prec -\beta I, \label{eq:appropriatelypermuted1}\\
A_{b}^{(k)}(t)^T
\begin{bmatrix}
\Gamma^{\Out}_k(t+1)&0\\
0&\Omega^{\Out}_k(t+1)
\end{bmatrix}A_{b}^{(k)}(t)
-\begin{bmatrix}
\Gamma^{\In}_k(t)&0\\
0&\Omega^{\In}_k(t)
\end{bmatrix}+C_{b}^{(k)}(t)^TC_{b}^{(k)}(t)\prec -\beta I. \label{eq:appropriatelypermuted2}
\end{align}
From these inequalities, we can infer that
\begin{align*}
A_r^{(k)}(t)\,\diag\left(
\Gamma^{(k)}(t),\Gamma_{\In}^{(k)}(t)\right)A_r^{(k)}(t)^T
-\,\diag\left(\Gamma^{(k)}(t+1),\Gamma_{\Out}^{(k)}(t+1)\right) + B_r^{(k)}(t)B_r^{(k)}(t)^T \prec -\beta I,\\
A_r^{(k)}(t)^T\,\diag\left(
\Gamma^{(k)}(t+1),\Gamma_{\Out}^{(k)}(t+1)\right)A_r^{(k)}(t)
-\,\diag\left(\Gamma^{(k)}(t),\Gamma_{\In}^{(k)}(t)\right)+ C_r^{(k)}(t)^T C_r^{(k)}(t) \prec -\beta I.
\end{align*}
Thus, system $G_r$ is strongly stable and the given realization $(A_r^{(k)}(t), B_r^{(k)}(t), C_r^{(k)}(t), D^{(k)}(t))$ is balanced with balanced generalized gramians $\Gamma^{(k)}(t)$ and $\Gamma^{(ij)}(t)$.
\end{proof}

\section{Error Bounds}\label{errorboundssection}
Next, we develop upper bounds on the $\ell_2$-induced norm of the error system $(G-G_r)$, which generalize their counterparts for single LTV systems in \begin{NoHyper}{\cite{rantzer2004}}\end{NoHyper} and single nonstationary LPV systems in \begin{NoHyper}{\cite{farhood2007}}\end{NoHyper}. For each $t\in \mathbb{N}_0$, we define $\hat{\Omega}(t)=\diag(\Omega^{(k)}(t))_{k\in V}$, and let $\tilde{\Omega}(t)$ be the block-diagonal augmentation of $\Omega^{(ij)}(t)$, where $(i,j)\in E$.  The specific ordering of the diagonal blocks in $\tilde{\Omega}(t)$ is inconsequential for our purposes. Then, we define
\[\bar{\Omega}(t)=\diag(\hat{\Omega}(t),\tilde{\Omega}(t)) \quad \quad \mbox{and} \quad \quad  \Omega=\diag(\bar{\Omega}(t))_{t\in\mathbb{N}_0}.\]
Depending on the states that are to be truncated,  some diagonal blocks of $\Omega$ may have zero dimensions and, hence, are nonexistent. For instance, suppose for some $(t_0,k_0)\in  \mathbb{N}_0\times V$ and $(i_0,j_0)\in E$,  we have $n^{(k_0)}(t_0)= r^{(k_0)}(t_0)$ and $n^{(i_0j_0)}(t_0)= r^{(i_0j_0)}(t_0)$. Then, the diagonal blocks $\Omega^{(k_0)}(t_0)$ and $\Omega^{(i_0j_0)}(t_0)$ have zero dimensions and do not appear in $\Omega$. Alternatively, if  $n^{(k)}(t)\ne r^{(k)}(t)$ and $n^{(ij)}(t)\ne r^{(ij)}(t)$ only for $t=t_0$, $k=k_0$, and $(i,j)=(i_0,j_0)$, then $\Omega=\diag(\Omega^{(k_0)}(t_0),\Omega^{(i_0j_0)}(t_0))$.
\vskip 3mm
\begin{thm}\label{errorboundomegaequaltoI}
If $\Omega^{(k)}(t)=I$ and $\Omega^{(ij)}(t)=I$, for all $(t,k)\in \mathbb{N}_0\times V$ and $(i,j) \in E$,  then $\|(G-G_r)\| <2$.
\end{thm}
\begin{proof}
To prove this result, we construct solutions to (\ref{eq:performanceinequality}) for a realization of the error system $\frac{1}{2}(G-G_r)$ and $\gamma=1$. Note that, since $G$ and $G_r$ are strongly stable, so is $\frac{1}{2}(G-G_r)$. We apply the Schur complement formula twice to (\ref{eq:appropriatelypermuted1}) and invoke (\ref{eq:appropriatelypermuted2}) to show that
\begin{equation}\label{ineq1_mp}
K^{(k)}(t)^TR_2^{(k)}(t+1)^{-1}K^{(k)}(t)-R_1^{(k)}(t) \prec - \beta I,
\end{equation}
\begin{align*}
&\mbox{where } K^{(k)}(t)=\begin{bmatrix}
\rule{0mm}{4.5mm}0&  0 & A_b^{(k)}(t)\\
\rule{0mm}{4.5mm}0& 0 & C_b^{(k)}(t)\\
\rule{0mm}{4.5mm}A_b^{(k)}(t) & B_b^{(k)}(t) & 0 \\
\end{bmatrix}, \quad R_1^{(k)}(t)=
\begin{bmatrix}
\Gamma^{\In}_k(t)^{-1} & 0 & 0 & 0 & 0\\
0 & \Omega^{\In}_k(t)^{-1} & 0 & 0 & 0\\
0 & 0 &I_{n_{u}^{(k)}(t)} & 0 & 0\\
0 & 0 & 0 & \Gamma^{\In}_k(t) & 0\\
0 & 0 & 0 & 0 & \Omega^{\In}_k(t)
\end{bmatrix},\\
&R_2^{(k)}(t+1)=
\begin{bmatrix}
\Gamma^{\Out}_k(t+1)^{-1} & 0 & 0 & 0 & 0\\
0 & \Omega^{\Out}_k(t+1)^{-1} & 0 & 0 & 0\\
0 & 0 &I_{n_{y}^{(k)}(t)} & 0 & 0\\
0 & 0 & 0 & \Gamma^{\Out}_k(t+1) & 0\\
0 & 0 & 0 & 0 & \Omega^{\Out}_k(t+1)
\end{bmatrix}.
\end{align*}
To simplify the algebraic manipulations when applying the Schur complement formula, we provisionally set the $-\beta I$ term in (\ref{eq:appropriatelypermuted1}) to zero. We do subsequently add some $-\beta I$ term to the right-hand side of (\ref{ineq1_mp}) just to  emphasize that the left-hand side is uniformly negative definite. We now pre- and post-multiply (\ref{ineq1_mp}) by $P^{(k)}(t)^T$  and $P^{(k)}(t)$, respectively, and insert $L^{(k)}(t)^TL^{(k)}(t)=I$ to get
\begin{multline*}
\left(L^{(k)}(t)K^{(k)}(t)P^{(k)}(t)\right)^T\left(L^{(k)}(t)R_2^{(k)}(t+1)^{-1}L^{(k)}(t)^{T}\right)\left(L^{(k)}(t)K^{(k)}(t)P^{(k)}(t)\right)\\
-P^{(k)}(t)^TR_1^{(k)}(t)P^{(k)}(t)  \prec - \beta I,
\end{multline*}
where $P^{(k)}(t)$ and $L^{(k)}(t)$ are, respectively, defined as
\[P^{(k)}(t)=\frac{1}{\sqrt{2}}\!\!\begin{bmatrix}
I & I & 0& 0        & 0\\
0 & 0 & I& 0        & I\\
0 & 0 & 0&\sqrt{2}I_{n_{u}^{(k)}(t)} & 0\\
-I& I & 0& 0        & 0\\
0 & 0 & I& 0        &-I
\end{bmatrix} \quad \quad \mbox{and} \quad \quad L^{(k)}(t)=\frac{1}{\sqrt{2}}\begin{bmatrix}
-I& 0 & 0         & I& 0\\
I & 0 & 0         & I& 0\\
0 & I & 0         & 0& I\\
0 & 0 & \sqrt{2}I_{n_{y}^{(k)}(t)} & 0& 0\\
0 & -I& 0         & 0& I
\end{bmatrix}.\]
$P^{(k)}(t)^{{T}}R_1^{(k)}(t)P^{(k)}(t)$ and $L^{(k)}(t)R_2^{(k)}(t+1)^{-1}L^{(k)}(t)^{{T}}$ have a similar structure:
$P^{(k)}(t)^{{T}}R_1^{(k)}(t)P^{(k)}(t)=$
\[
\frac{1}{2}\begin{bmatrix}
\Gamma^{\In}_k(t)^{-1}+\Gamma^{\In}_k(t) & \rule{0mm}{4.5mm}\Gamma^{\In}_k(t)^{-1}-\Gamma^{\In}_k(t) & 0 & 0 & 0\\
\Gamma^{\In}_k(t)^{-1}-\Gamma^{\In}_k(t) & \rule{0mm}{4.5mm}\Gamma^{\In}_k(t)^{-1}+\Gamma^{\In}_k(t) & 0 & 0 & 0\\
0 & 0 & \Omega^{\In}_k(t)^{-1}+\Omega^{\In}_k(t) &  \rule{0mm}{4.5mm} 0&\Omega^{\In}_k(t)^{-1}-\Omega^{\In}_k(t)\\
0 & 0 & 0 & \rule{0mm}{4.5mm}2I_{n_{u}^{(k)}(t)}& 0\\
0 & 0 & \Omega^{\In}_k(t)^{-1}-\Omega^{\In}_k(t) &  \rule{0mm}{4.5mm} 0&\Omega^{\In}_k(t)^{-1}+\Omega^{\In}_k(t)
\end{bmatrix},\]
$L^{(k)}(t)R_2^{(k)}(t+1)^{-1}L^{(k)}(t)^{{T}}=\frac{1}{2}\times$
\[
\begin{bmatrix}
\Gamma^{\Out}_k(t{+}1)^{{-}1}{+}\Gamma^{\Out}_k(t{+}1) & \rule{0mm}{4.5mm}\Gamma^{\Out}_k(t{+}1)^{{-}1}{-}\Gamma^{\Out}_k(t{+}1) & \!\!\!\!\!\!\!\!\!0 & \!\!\!\!\!\!\! 0 & \!\!\!\!\!\!\! 0\\
\Gamma^{\Out}_k(t{+}1)^{{-}1}{-}\Gamma^{\Out}_k(t{+}1) & \rule{0mm}{4.5mm}\Gamma^{\Out}_k(t{+}1)^{{-}1}{+}\Gamma^{\Out}_k(t{+}1) & \!\!\!\!\!\!\!\!\!0 & \!\!\!\!\!\!\! 0 & \!\!\!\!\!\!\! 0\\
0 &  0 &   \!\!\!\!\!\!\!\!\! \Omega^{\Out}_k(t{+}1)^{{-}1}{+}\Omega^{\Out}_k(t{+}1) &  \!\!\!\!\!\!\! \rule{0mm}{4.5mm} 0       & \!\!\!\!\!\!\! \Omega^{\Out}_k(t{+}1)^{{-}1}{-}\Omega^{\Out}_k(t{+}1)\\
0 &  0 & \!\!\!\!\!\!\!\!\! 0 & \!\!\!\!\!\!\! \rule{0mm}{4.5mm}2I_{\!n_{y}^{(k)}\!(t)}  &  \!\!\!\!\!\!\! 0\\
0 &  0 & \!\!\!\!\!\!\!\!\! \Omega^{\Out}_k(t{+}1)^{{-}1}{-}\Omega^{\Out}_k(t{+}1)   &  \!\!\!\!\!\!\! \rule{0mm}{4.5mm} 0       & \!\!\!\!\!\!\! \Omega^{\Out}_k(t{+}1)^{{-}1}{+}\Omega^{\Out}_k(t{+}1)
\end{bmatrix}\!\!{.}\]
For some appropriately defined $N_{12}^{(k)}(t)$ and $N_{21}^{(k)}(t)$, $L^{(k)}(t)K^{(k)}(t)P^{(k)}(t)=\begin{bmatrix}
M^{(k)}(t)& N_{12}^{(k)}(t)\\
N_{21}^{(k)}(t)& \bar{A}_{22}^{(k)}(t)
\end{bmatrix}$, where
\begin{equation*}
M^{(k)}(t)=\left[\begin{array}{cc|c}
\rule{0mm}{5mm}A_r^{(k)}(t)                         & \quad\,\,     0                &\,\,\frac{1}{\sqrt{2}}B_r^{(k)}(t)\\
\rule{0mm}{5mm}      0                      & \quad\,\,A_b^{(k)}(t)           & \,\,\frac{1}{\sqrt{2}}B_b^{(k)}(t)\\
\hline
\rule{0mm}{5mm}\frac{-1}{\sqrt{2}}C_r^{(k)}(t)  &\quad\,\,\frac{1}{\sqrt{2}} C_b^{(k)}(t)& \,\,0
\end{array}\right].
\end{equation*}
The matrices in $M^{(k)}(t)$ can be used to describe the dynamics of the error system $\frac{1}{2}(G-G_r)$. However, the resultant system equations will not be in the form of (\ref{eq:distributedsystemequations}), but can be equivalently expressed in that form through the use of appropriate permutations. Since $\Omega^{(k)}(t)=I$ and $\Omega^{(ij)}(t)=I$, for all $(t,k)\in \mathbb{N}_0\times V$ and $(i,j) \in E$, then it is not difficult to see that
\[
M^{(k)}(t)^T\begin{bmatrix}V_2^{(k)}(t+1)&0 \\0 &I\end{bmatrix}M^{(k)}(t)
-\begin{bmatrix}
V_1^{(k)}(t) & 0\\
0 & I
\end{bmatrix} \prec -\beta I,\]
\begin{align*}
V_2^{(k)}(t+1)&=
\frac{1}{2}\begin{bmatrix}
\rule{0mm}{4.5mm}\left(\Gamma^{\Out}_k(t+1)^{-1}+\Gamma^{\Out}_k(t+1)\right) &  \left(\Gamma^{\Out}_k(t+1)^{-1}-\Gamma^{\Out}_k(t+1)\right)&  0 \\
\rule{0mm}{4.5mm}\left(\Gamma^{\Out}_k(t+1)^{-1}-\Gamma^{\Out}_k(t+1)\right) &  \left(\Gamma^{\Out}_k(t+1)^{-1}+\Gamma^{\Out}_k(t+1)\right)&  0 \\
\rule{0mm}{4.5mm}0&  0&  2I
\end{bmatrix} \succ \beta I,\\
V_1^{(k)}(t)&=
\frac{1}{2}\begin{bmatrix}
\rule{0mm}{4.5mm}\left(\Gamma^{\In}_k(t)^{-1}+\Gamma^{\In}_k(t)\right) & \,\, \left(\Gamma^{\In}_k(t)^{-1}-\Gamma^{\In}_k(t)\right)& \,\, 0 \\
\rule{0mm}{4.5mm}\left(\Gamma^{\In}_k(t)^{-1}-\Gamma^{\In}_k(t)\right) & \,\, \left(\Gamma^{\In}_k(t)^{-1}+\Gamma^{\In}_k(t)\right)& \,\, 0 \\
\rule{0mm}{4.5mm}0& \,\, 0& \,\, 2I
\end{bmatrix} \succ \beta I.
\end{align*}
That is, $V_2^{(k)}(t+1)$ and $V_1^{(k)}(t)$ correspond to the upper left corner blocks of $L^{(k)}(t)R_2^{(k)}(t+1)^{-1}L^{(k)}(t)^T$ and $P^{(k)}(t)^T R_1^{(k)}(t)P^{(k)}(t)$, respectively. Applying the appropriate permutations and invoking Lemma \ref{performancelemma} with $\gamma=1$, we get $\|\frac{1}{2}(G-G_r)\| <1$.
\end{proof}
\vskip 3mm
\begin{thm} \label{errorboundgeneralomega}
The error system $(G-G_r)$ resulting from balanced truncation satisfies
$\|(G-G_r)\|  < 2\zeta(\Omega)$, where $\zeta(X)$ denotes the sum of distinct diagonal entries of a square, possibly infinite dimensional, matrix $X$.
\end{thm}
\begin{proof}
As a truncated realization is itself balanced by Lemma \ref{reducedrealizationstableandbalanced}, the truncation procedure can be implemented in multiple steps. For each step $i$, we find the smallest diagonal entry in $\Omega$, which we denote by $q_{i}$, and truncate all the state variables with a corresponding diagonal entry in $\Omega$ equal to $q_{i}$. We then update $\Omega$ by removing all the entries equal to $q_i$. The resulting  reduced order system is denoted by $G_{r,i}$. Suppose that, in the first step, $\Omega^{(k)}(t)=q_1I$ for some $(t,k) \in \mathbb{N}_0\times V$ and/or $\Omega^{(ij)}(t)=q_1I$ for some $(i,j)\in E$ and (perhaps different) $t\in \mathbb{N}_0$. We want to show that $\|(G-G_{r,1})\|<2 q_1$. To do so, we construct a scaled system $G_{\new}$ with realization $(A^{(k)}(t), \frac{1}{\sqrt{q_1}}B^{(k)}(t), \frac{1}{\sqrt{q_1}}C^{(k)}(t), \frac{1}{q_1}D^{(k)}(t))$.

It is not difficult to verify that the corresponding $\Omega_{\new}^{(k)}(t)=I$ and/or $\Omega_{\new}^{(ij)}(t)=I$. Denote by $G_{\new,r,1}$ the reduced order system obtained after applying BT to $G_{\new}$. From Theorem \ref{errorboundomegaequaltoI}, $\|(G_{\new}-G_{\new,r,1})\|<2$. But, $\|(G_{\new}-G_{\new,r,1})\|=\frac{1}{q_1}\|(G-G_{r,1})\|$. And so, $\|(G-G_{r,1})\|<2 q_1$. The same procedure is applied in the following step to obtain  $\|(G_{r,1}-G_{r,2})\|<2 q_2$.  Then, by the triangle inequality, $\|(G-G_{r,2})\|=\|(G-G_{r,1}+G_{r,1}-G_{r,2})\|\le\|(G-G_{r,1})\|+\|(G_{r,1}-G_{r,2})\|< 2(q_1+q_2)$, and so on.
\end{proof}

Theorem \ref{errorboundgeneralomega} gives an upper bound on $\|(G-G_r)\|$, which may not always be finite as there may be infinitely many distinct entries in $\Omega$. Next, we derive a tighter expression for the error bound which applies when the diagonal entries of $\Omega^{(k)}(t)$ and $\Omega^{(ij)}(t)$ define monotonic sequences in time. We define the subsets of time at which truncation occurs as $\mathcal{F}_k=\{t \in \mathbb{N}_0 \,|\, n^{(k)}(t)\neq r^{(k)}(t)\}$ and $\mathcal{F}_{(ij)}=\{t\in\mathbb{N}_0 \,|\, n^{(ij)}(t)\neq r^{(ij)}(t)\}$, for all $k\in V$ and $(i,j)\in E$.  Definition \ref{holdruledefinition} is from \begin{NoHyper}{\cite{farhood2007}}\end{NoHyper}.
\vskip 3mm
\begin{dfn}\label{holdruledefinition}
Consider a scalar sequence $\alpha_t$ which is defined on some subset $\mathcal{W}$ of $\mathbb{N}_0$, and let $t_{\min}=\min\{t\,|\,t\in \mathcal{W}\}$.  We extend the domain of definition of $\alpha_t$ to all $t\in \mathbb{N}_0$ by defining the following rule:
\[
    \alpha_t{=}\left\{
                \begin{array}{ll}
                 \alpha_{t_{\min}} & \mbox{if} \   0\le t   \le t_{\min},\\
                 \alpha_d          & \mbox{if} \  t_{\min} < t, \  \mbox{where } \, d=\max\{\tau\le t\,|\,\tau \in \mathcal{W}\}.
                \end{array}
              \right.
  \]
\end{dfn}
\vskip 3mm
\begin{thm} \label{monotonicerrorbound}
For all $k\in V$ and $(i,j)\in E$, if $\Omega^{(k)}(t)=w^{(k)}(t)I$ for $t\in \mathcal{F}_k$, and  $\Omega^{(ij)}(t)=w^{(ij)}(t)I$ for $t\in \mathcal{F}_{(ij)}$, where the sequences $w^{(k)}(t)$ and $w^{(ij)}(t)$ are monotonic in time, then
\[\|(G-G_r)\|  < 2 \left(\sum_{k\in V} \sup_{t\in \mathcal{F}_{k}} w^{(k)}(t)+\sum_{(i,j)\in E}\sup_{t\in \mathcal{F}_{(ij)}}w^{(ij)}(t)\right).\]
\end{thm}
\begin{proof}
We need to prove the result for the special case where only one temporal state or one spatial state is truncated. The general case considered in the theorem then follows by repeated application of the result for this special case. Namely, we fix $k=k_0$, and assume that the corresponding temporal state is the only truncated state. Without loss of generality, we assume that $w^{(k_0)}(t) \le 1$ for all $t \in \mathcal{F}_{k_0}$, as this can be always achieved by scaling. We extend the domain of $w^{(k_0)}(t)$ to all $t \in \mathbb{N}_0$ using the rule in Definition \ref{holdruledefinition}. Clearly, the extended sequence is still monotonic. We first consider the case where $w^{(k_0)}(t)$ is monotone nondecreasing. For all $(t,k)\in\mathbb{N}_0\times V$, we define the state-space transformations $T_{\pre}^{(k)}(t)=w^{(k_0)}(t)^{-1/2}I$  and $T_{\post}^{(k)}(t)=w^{(k_0)}(t)^{1/2}I$. Note that $T_{\pre}^{(k)}(t)$ is bounded since $\Sigma^{(k_0)}(t)\succ \beta I$, for some $\beta > 0$ and all $t \in \mathbb{N}_0$. Then, we define a new realization $(A_{\new}^{(k)}(t), B_{\new}^{(k)}(t), C_{\new}^{(k)}(t), D^{(k)}(t))$ for system $G$, where
\[
A_{\new}^{(k)}(t)=T_{\pre}^{(k)}(t+1)A^{(k)}(t)T_{\post}^{(k)}(t),\quad \quad B_{\new}^{(k)}(t)=T_{\pre}^{(k)}(t+1)B^{(k)}(t), \quad \,\, \mbox{and} \,\, \quad C_{\new}^{(k)}(t)=C^{(k)}(t)T_{\post}^{(k)}(t).\]

For simplicity, we refer to system $G$ with the new realization as $G_{\new}$. We now show that the realization of $G_{\new}$ is balanced.  Recall that $\Sigma^{(k)}(t)$ and $\Sigma^{(ij)}(t)$ satisfy (\ref{eq:lyapinequ1}) and (\ref{eq:lyapinequ2}). We pre- and post-multiply (\ref{eq:lyapinequ1}) by $T_{\pre}^{(k)}(t+1)$, insert $T_{\post}^{(k)}(t)T_{\post}^{(k)}(t)^{-1}=T_{\post}^{(k)}(t)^{-1}T_{\post}^{(k)}(t)=I$ as needed, and define $\Sigma_{\new}^{(k)}(t)=w^{(k_0)}(t)^{-1}\Sigma^{(k)}(t)$ and $\Sigma_{\new}^{(ij)}(t)=w^{(k_0)}(t)^{-1}\Sigma^{(ij)}(t)$, to get
\[A_{\new}^{(k)}(t)\begin{bmatrix}
\Sigma_{\new}^{(k)}(t)&0\\
0&\Sigma_{\new,\In}^{(k)}(t)
\end{bmatrix}A_{\new}^{(k)}(t)^T
-\begin{bmatrix}
\Sigma_{\new}^{(k)}(t+1)&0\\
0&\Sigma_{\new,\Out}^{(k)}(t+1)
\end{bmatrix}+B_{\new}^{(k)}(t)B_{\new}^{(k)}(t)^T\prec -\beta I.\]
We also pre- and post-multiply (\ref{eq:lyapinequ2}) by $T_{\post}^{(k)}(t)$ and insert $T_{\pre}^{(k)}(t+1)^{-1}T_{\pre}^{(k)}(t+1)=T_{\pre}^{(k)}(t+1)T_{\pre}^{(k)}(t+1)^{-1}=I$ to get
\[
w^{(k_0)}(t+1)A_{\new}^{(k)}(t)^T\begin{bmatrix}
\Sigma^{(k)}(t+1)&\!\!\!0\\
0&\!\!\!\Sigma_{\Out}^{(k)}(t+1)
\end{bmatrix}A_{\new}^{(k)}(t)
-w^{(k_0)}(t)\begin{bmatrix}
\Sigma^{(k)}(t)&\!\!\!0\\
0&\!\!\!\Sigma_{\In}^{(k)}(t)
\end{bmatrix}+C_{\new}^{(k)}(t)^TC_{\new}^{(k)}(t)\prec -\beta I.\]
Since  $0<w^{(k_0)}(t)\le 1$, then $w^{(k_0)}(t)^{-1}\ge w^{(k_0)}(t)$. Also, since $w^{(k_0)}(t)$ is monotone nondecreasing, then $w^{(k_0)}(t)\le w^{(k_0)}(t{+}1)$ and $w^{(k_0)}(t)^{-1}\ge w^{(k_0)}(t{+}1)^{-1}{.}$ With this in mind, it is not difficult to verify that
\begin{multline*}
w^{(k_0)}(t+1)^{-1}A_{\new}^{(k)}(t)^T
\begin{bmatrix}
\Sigma^{(k)}(t+1)&0\\
0&\Sigma_{\Out}^{(k)}(t+1)
\end{bmatrix}A_{\new}^{(k)}(t)\\
-w^{(k_0)}(t)^{-1}\begin{bmatrix}
\Sigma^{(k)}(t)&0\\
0&\Sigma_{\In}^{(k)}(t)
\end{bmatrix}+C_{\new}^{(k)}(t)^TC_{\new}^{(k)}(t)\prec -\beta I.
\end{multline*}

Thus, $\Sigma_{\new}^{(k)}(t)$ and $\Sigma_{\new}^{(ij)}(t)$ satisfy $(\ref{eq:lyapinequ1})$ and $(\ref{eq:lyapinequ2})$ for the realization of $G_{\new}$, i.e., the realization of $G_{\new}$ is balanced and can be reduced by BT. We denote by $G_{\new,r}$ the reduced order system obtained by truncating the temporal state of subsystem $G_{\new}^{(k_0)}$. Since $\Omega^{(k_0)}(t)=w^{(k_0)}(t)I$, then $\Omega_{\new}^{(k_0)}(t)=I$  for $t\in \mathcal{F}_{k_0}$. By Theorem \ref{errorboundomegaequaltoI}, $\|(G_{\new}-G_{\new,r})\|<2$. But, because of the special structure of $T_{\post}^{(k)}(t)$ and $T_{\pre}^{(k)}(t)$, system $(G_{\new}-G_{\new,r})$ is equivalent to $(G-G_r)$, and so, $\|(G-G_{r})\|<2$. A similar argument holds for the case where $w^{(k_0)}(t)$ is monotone nonincreasing, with the state-space transformations defined as $T_{\pre}^{(k)}(t)=w^{(k_0)}(t)^{1/2}I$  and $T_{\post}^{(k)}(t)=w^{(k_0)}(t)^{-1/2}I$, for all $(t,k)\in\mathbb{N}_0\times V$.
\end{proof}

For example, suppose that the only truncated state is $x^{(k_0)}$, the truncation only occurs at three time-steps $t_1$, $t_2$, $t_3$, and  $\Omega^{(k_0)}(t_1)=\diag(7,3,2)$, $\Omega^{(k_0)}(t_2)=\diag(4,2,2)$, $\Omega^{(k_0)}(t_3)=\diag(6,5,2)$. The bound from Theorem \ref{errorboundgeneralomega}  gives $2\times(7+6+5+4+3+2)=54$, whereas the bound from Theorem \ref{monotonicerrorbound} gives $2\times(2+5+7)=28$. As can be seen below, in the latter case, the truncation sequences are  $\{2,\diag(2,2),2\}$, $\{3,4,5\}$, and $\{7,6\}$. That is, we first truncate the terms in squares, then we truncate the circled terms, and finally we truncate the terms in triangles. In the first truncation sequence, all truncated entries are equal to $2$. The second and third truncation sequences are monotone (increasing and decreasing, respectively) with the largest entries equal to $5$ and $7$. This gives the error bound $2\times(2+5+7)=28$ computed above.
\[
\Omega^{(k_0)}(t_1)=\begin{bmatrix}\triangled{7}&0&0\\0&\circled{3}&0\\0&0&\boxed{2}\end{bmatrix},\quad \quad
\Omega^{(k_0)}(t_2)=\begin{bmatrix}\circled{4}&0&0\\0&\boxed{2}&0\\0&0&\boxed{2}\end{bmatrix},    \quad \quad
\Omega^{(k_0)}(t_3)=\begin{bmatrix}\triangled{6}&0&0\\0&\circled{5}&0\\0&0&\boxed{2}\end{bmatrix}.\]

\section{Eventually Time-Periodic Systems}\label{ETPsection}
In this section, we deal with eventually time-periodic (ETP) subsystems, i.e., subsystems with state-space matrices which become time-periodic after some initial amount of time, and show that the bound given in Theorem \ref{errorboundgeneralomega} reduces to a finite sum. A distributed system $G$ is said to be $(h,q)$-ETP, for some integers $h \geq 0$ and $q >0$, if the state-space matrices of the subsystems are $(h,q)$-ETP, e.g., $A^{(k)}(t+h+zq)=A^{(k)}(t+h)$, for all $t,z \in \mathbb{N}_0$ and $k\in V$. The class of ETP subsystems includes as special cases time-periodic subsystems ($(0,q)$-ETP) and finite time-horizon subsystems ($(h,1)$-ETP with zero state-space matrices for $t\geq h$).
\vskip 3mm
\begin{lmm} \label{periodicsolutionstostability}
For a strongly stable, $q$ time-periodic system G, there exist $q$ time-periodic solutions to (\ref{eq:strongstabilitycondition}).
\end{lmm}

\begin{proof}
The proof can be found in \begin{NoHyper}{\cite{DAJMFCDC15}}\end{NoHyper}, and is included here for the sake of completion. Since G is strongly stable, then there exist solutions to (\ref{eq:strongstabilitycondition}), which we denote by $W^{(k)}(t)$ and $W^{(ij)}(t)$. From these solutions, we construct $q$ time-periodic solutions, which we denote $W^{(k)}_{\per}(t)$ and $W^{(ij)}_{\per}(t)$. To do so, we resort to averaging techniques similar to the ones used in \begin{NoHyper}{\cite{geir1999}}\end{NoHyper} and \begin{NoHyper}{\cite{farhood2012}}\end{NoHyper}. Since the distributed system is $q$ time-periodic, then $A^{(k)}(t+zq)=A^{(k)}(t)$ for all $t, z\in \mathbb{N}_0$ and $k\in V$. We fix $t$ in $\mathbb{N}_0$, choose an integer $\lambda \ge 1$, and evaluate (\ref{eq:strongstabilitycondition}) at $(t+zq,k)$ for $z=0,\ldots,\lambda-1$. Averaging the resulting inequalities, we obtain
\[
A^{(k)}(t)^T \begin{bmatrix}
Y_{\lambda}^{(k)}(t+1)& 0\\
0 &Y_{\Out,\lambda}^{(k)}(t+1)
\end{bmatrix}A^{(k)}(t)
-\begin{bmatrix}
Y_{\lambda}^{(k)}(t)& 0\\
0 & Y_{\In,\lambda}^{(k)}(t)
\end{bmatrix} \prec -\beta I,
\]
where $Y_{\lambda}^{(k)}(t)=\frac{1}{\lambda}\sum_{z = 0}^{\lambda -1}W^{(k)}(t+zq)$. $Y_{\lambda}^{(ij)}(t)$ are defined similarly. Since the solutions to (\ref{eq:strongstabilitycondition}) are uniformly bounded, then so are $Y_{\lambda}^{(k)}(t)$ and $Y_{\lambda}^{(ij)}(t)$. Then, there exist weakly convergent subsequences $Y_{\lambda_c}^{(k)}(t)$ and $Y_{\lambda_c}^{(ij)}(t)$ with limits $L^{(k)}(t)$ and $L^{(ij)}(t)$, respectively. The reader is referred to \begin{NoHyper}{\cite{kubrusly2011}}\end{NoHyper} for further details on convergence in weak topology. By construction, the limits are positive definite and satisfy $L^{(k)}(t) \succ \beta I$ and $L^{(ij)}(t)\succ \beta I$. It remains to be shown that they are $q$ time-periodic. This is done for $L^{(k)}(t)$. The proof for $L^{(ij)}(t)$ follows similarly.
\begin{align*}
L^{(k)}(t+q)-L^{(k)}(t)&=\lim_{\lambda_c \rightarrow \infty}\frac{1}{\lambda_c}\sum_{z = 0}^{\lambda_c -1}\left(W^{(k)}(t+(z+1)q)-W^{(k)}(t+zq)\right)\\
&=\lim_{\lambda_c \rightarrow \infty}\frac{1}{\lambda_c}\left( W^{(k)}(t+\lambda_c q)-W^{(k)}(t)\right)=0.
\end{align*}
We set $W^{(k)}_{\per}(t)=L^{(k)}(t)$ and $W^{(ij)}_{\per}(t)=L^{(ij)}(t)$, for all $(t,k)\in \mathbb{N}_0\times V$ and $(i,j) \in E$.
\end{proof}
\vskip 3mm
\begin{thm} \label{eventuallyperiodicsolutionstostability}
For a strongly stable, $(h,q)$-ETP system $G$, there exist $(h,q)$-ETP solutions  $X^{(k)}_{\eper}(t)$, $X^{(ij)}_{\eper}(t)$ and $Y^{(k)}_{\eper}(t)$, $Y^{(ij)}_{\eper}(t)$ to (\ref{eq:lyapinequ1}) and (\ref{eq:lyapinequ2}), respectively, for all $(t,k)\in \mathbb{N}_0\times V$ and $(i,j) \in E$.
\end{thm}

\begin{proof}
Since $G$ is strongly stable, then there exist solutions to (\ref{eq:strongstabilitycondition}), which we denote by $W^{(k)}(t)$ and $W^{(ij)}(t)$. From these solutions, we construct $(h,q)$-ETP solutions to (\ref{eq:strongstabilitycondition}) as shown next. We will only need the LMIs that correspond to $t \ge h-1$. So, without loss of generality, we assume  $h=1$. For all $k \in V$, we have
\begin{align}
&A^{(k)}(0)^{{T}}\begin{bmatrix}
W^{(k)}(1)&0\\
0&W_{\Out}^{(k)}(1)
\end{bmatrix}A^{(k)}(0) - \begin{bmatrix}
W^{(k)}(0)&0\\
0&W_{\In}^{(k)}(0)\end{bmatrix}\prec -\beta I, \nonumber\\
&A^{(k)}(t)^T \begin{bmatrix}
W^{(k)}(t+1)&0\\
0&W_{\Out}^{(k)}(t+1)
\end{bmatrix}A^{(k)}(t)- \begin{bmatrix}
W^{(k)}(t)&0\\
0&W_{\In}^{(k)}(t)\end{bmatrix} \prec -\beta I, \quad \quad t \ge 1. \label{eq:tobeperiodicportion}
\end{align}
By Lemma \ref{periodicsolutionstostability}, there exist $q$ time-periodic solutions to (\ref{eq:tobeperiodicportion}), denoted by $W^{(k)}_{\per}(t)$ and $W^{(ij)}_{\per}(t)$. We can always choose $\alpha>0$ such that $\alpha \, W^{(k)}_{\per}(1)\prec  W^{(k)}(1)$ and $\alpha \, W^{(ij)}_{\per}(1) \prec W^{(ij)}(1)$. Then, the following holds
\[
A^{(k)}(0)^T \begin{bmatrix}
\alpha W^{(k)}_{\per}(1)&0\\
0&\alpha W^{(k)}_{\Out,\per}(1)\end{bmatrix}A^{(k)}(0)
- \begin{bmatrix}
W^{(k)}(0)&0\\
0&W_{\In}^{(k)}(0)
\end{bmatrix} \prec -\beta I.\]
We set $W^{(k)}_{\eper}(0)=W^{(k)}(0)$ and $W^{(k)}_{\eper}(t)=\alpha \, W^{(k)}_{\per}(t)$ for $t\ge 1$. $W^{(ij)}_{\eper}(t)$ are defined similarly.

Since $W^{(k)}_{\eper}(t)$ and $W^{(ij)}_{\eper}(t)$ are $(h,q)$-ETP solutions to (\ref{eq:strongstabilitycondition}), we can choose $\mu>0$ such that $Y^{(k)}_{\eper}(t)=\mu \, W^{(k)}_{\eper}(t) $ and $Y^{(ij)}_{\eper}(t)=\mu \,W^{(ij)}_{\eper}(t)$ are solutions to (\ref{eq:lyapinequ2}). Also, by applying the Schur complement formula twice to (\ref{eq:strongstabilitycondition}), we see that we can choose $\xi >0$ such that $X^{(k)}_{\eper}(t)=\xi \, W^{(k)}_{\eper}(t)^{-1} $ and $X^{(ij)}_{\eper}(t)=\xi \,W^{(ij)}_{\eper}(t)^{-1}$ are solutions to (\ref{eq:lyapinequ1}).
\end{proof}

\vskip 3mm
\begin{crly}\label{eventuallyperiodicerrorbound}
A strongly stable, $(h,q)$-ETP system $G$ has an $(h,q)$-ETP balanced realization and $(h,q)$-ETP generalized gramians   $\Sigma_{\eper}^{(k)}(t)=\diag(\Gamma_{\eper}^{(k)}(t),\Omega_{\eper}^{(k)}(t))$ and $\Sigma_{\eper}^{(ij)}(t)=\diag(\Gamma_{\eper}^{(ij)}(t),\Omega_{\eper}^{(ij)}(t))$. Moreover, system $G_r$, resulting from BT, has an $(h,q)$-ETP balanced realization and satisfies $\|(G-G_r)\|  < 2\zeta(\Omega_{\eper})$, where $\Omega_{\eper}$ is the $(h,q)$-truncation of $\Omega$ defined in Section \ref{errorboundssection}, i.e., $\Omega_{\eper}=\diag(\bar{\Omega}_{\eper}(t))_{t\in\{0,1,\ldots,h+q-1\}}$.
\end{crly}

The bound in Corollary \ref{eventuallyperiodicerrorbound} may be further improved by using Theorem \ref{monotonicerrorbound} to compute the bound due to the states truncated over the finite time-horizon, i.e., for $0\le t < h$.

\section{Illustrative Example}\label{examplesection}
In this section, we apply the BT method to a distributed system $G$ with $N=5$ agents interconnected as in Figure \ref{fig:digraphillustration}. The temporal states and the spatial states are of constant dimensions $n_T=6$ and $n_S=3$, respectively.  There are two sets of building blocks for the state-space matrices: one for the odd-numbered subsystems and another for the even-numbered subsystems. All the blocks of the state-space matrices are constants, except for the $A_{00}^{(k)}(t)$ terms which are $(h=0,q=28)$-ETP, i.e., satisfy $A_{00}^{(k)}(t+28\,z){=}A_{00}^{(k)}(t)$, for all $t,z \in \mathbb{N}_0$ and $k\in V$. Namely, $A_{00}^{(k)}(t)=A_{TT}$ for $t=0,\ldots,6$, $A_{00}^{(k)}(t)=\mathcal{M}A_{TT}\mathcal{M}^T$ for $t=7,\ldots,13$, $A_{00}^{(k)}(t)=\mathcal{M}^2 A_{TT}(\mathcal{M}^T)^2$ for $t=14,\ldots,20$, and $A_{00}^{(k)}(t)=\mathcal{M}^3A_{TT}(\mathcal{M}^T)^3$ for $t=21,\ldots,27$,  where $A_{TT}$ and $\mathcal{M}$ are building blocks. For odd-numbered subsystems,
\begin{align*}
&A_{TT}{=}0.1\!\!\left[\!\!\begin{array}{cc}
\begin{bmatrix}
 -9   & -7 \\
 -5   & -5
\end{bmatrix} & \,\,
\begin{bmatrix}
0.1  &  0.3   & -0.1   &  \,\,\,\,\,0.2\\
0.3  &  0.2   & \,\,\,\,\, 0.1   & -0.2
\end{bmatrix}\\
\rule{0mm}{7mm}
\begin{bmatrix}
\diag(1,-2)\\
-I_2
\end{bmatrix}&\,\,0.01\diag(-5,1,-3,2)
\end{array}\!\!\right]\!\!{,}\,\,A_{TS}{=}0.1\!\!\left[\!\!\!\begin{array}{c}
\begin{bmatrix}
   -0.5  & \,\,\,\,\, 0.5  & \,\,\,\,\,0.01\\
\,\,\,\,\,0.5  & -0.5  & -0.02
\end{bmatrix}\\
0_{4\times3}
\end{array}\!\!\!
\right]\!\!{,} \,\, B_{T}{=}0.2\begin{bmatrix}
I_2\\
0_{4\times 2}
\end{bmatrix}\!\!{,}\\
&A_{ST}=0.1\!\begin{bmatrix}
\diag(1{,}{-}2{,}0.1)&    \diag(0.5{,}0.4{,}0.2)
\end{bmatrix}\!{,}\,\,A_{SS}{=}0_{3 \times 3}{,}\,\,
B_{S}{=}0.1\begin{bmatrix}
I_2\\
0_{1\times2}
\end{bmatrix}, \,\,C_T=\begin{bmatrix}
I_2 & 0_{2\times 4}
\end{bmatrix},\,\,C_S=0_{2 \times 3}.
\end{align*}
Subscripts $T$  and $S$ refer to temporal and spatial terms, respectively. E.g., for $k=1$ and all $t \in \mathbb{N}_0$,
\begin{align*}
&A_{02}^{(1)}(t)=A_{03}^{(1)}(t)=A_{04}^{(1)}(t)=A_{TS},\quad B_{0}^{(1)}(t)=B_{T}, \quad C_{0}^{(1)}(t)=C_{T}, \quad B_{2}^{(1)}(t)=B_{3}^{(1)}(t)=B_{5}^{(1)}(t)=B_{S},\\
& C_{2}^{(1)}(t)=C_{3}^{(1)}(t)=C_{4}^{(1)}(t)=C_{S}, \quad \quad A_{20}^{(1)}(t)=A_{30}^{(1)}(t)=A_{50}^{(1)}(t)=A_{ST}, \quad \quad \mbox{etc.}
\end{align*}
$\mathcal{M}$ is given by $\mathcal{M}=\begin{bmatrix}
0 &0 &0 &0 &1 &0\\
0 &0 &1 &0 &0 &0\\
1 &0 &0 &0 &0 &0\\
0 &0 &0 &0 &0 &1\\
0 &0 &0 &1 &0 &0\\
0 &1 &0 &0 &0 &0\end{bmatrix}$. As for even-numbered subsystems, $A_{SS}=0_{3 \times 3}$, $B_S{=}0_{3 \times 2}$,
\begin{align*}
&A_{TS}=0.1\!\!\begin{bmatrix}
\diag({-}0.5,0.1,{-}0.1)\\
\diag({-}0.5,{-}0.2,0.3)
\end{bmatrix}\!\!{,} \,\,\, C_S=\begin{bmatrix}
{-}I_2 & 0_{2 \times 1}
\end{bmatrix}\!\!{,} \,\,\, A_{ST}=\!\begin{bmatrix}
0.2I_3 &\!\!{-}0.03I_3
\end{bmatrix}\!{,} \,\,\, C_T=\begin{bmatrix}
0& 1 & 0 & 0 & 0 & 0\\
0& 0 & 0 & 1 & 0 & 0
\end{bmatrix}\!{,}\\
&B_{T}=0.1\!\begin{bmatrix}
1 & 0\\
0 & 0\\
0 & 1\\
0 & 0\\
0 & 0\\
0 & 0
\end{bmatrix}\!{,} \quad A_{TT}=0.1\!\!\begin{bmatrix}
  \,\,\,\,\,  1   &   -4             & -0.3             & \,\,\,\,\, 0.1  &  \,\,\,\,\, 0.5  & \,\,\,\,\, 0.3\\
  \,\,\,\,\,  3   &   -5             &  \,\,\,\,\, 0.2  & \,\,\,\,\, 0.1  & -0.2             & \,\,\,\,\, 0.3\\
  \,\,\,\,\,  0.1 &   -0.3           & -0.5             & \,\,\,\,\, 0.2  &  \,\,\,\,\, 0.1  & \,\,\,\,\, 0.1\\
   -0.2           &   \,\,\,\,\, 0   &  \,\,\,\,\, 0    & -0.1            &  \,\,\,\,\, 0    & \,\,\,\,\, 0  \\
  \,\,\,\,\,  0   &   \,\,\,\,\, 0.1 &  \,\,\,\,\, 0    & \,\,\,\,\, 0    &  \,\,\,\,\, 0.15 & \,\,\,\,\, 0  \\
  \,\,\,\,\,  0   &   \,\,\,\,\, 0   &  \,\,\,\,\, 0.3  & \,\,\,\,\, 0    &  \,\,\,\,\, 0    &-0.1
\end{bmatrix}\!\!{,} \quad \mathcal{M}=\!\begin{bmatrix}
 0 &0 &0 &1 &0  &0\\
 0 &0 &0 &0 &-1 &0\\
 0 &1 &0 &0 &0  &0\\
-1 &0 &0 &0 &0  &0\\
 0 &0 &0 &0 &0  &1\\
 0 &0 &1 &0 &0  &0
 \end{bmatrix}\!{.}
\end{align*}

For all $k\in V$ and $t\in \mathbb{N}_0$, $D^{(k)}(t)=0_{2\times2}$.
Using Lemma \ref{performancelemma}, we show that system $G$ is strongly stable and find an upper bound $\gamma$ on $\|G\|$. Namely, we find $(0,28)$-ETP solutions to (\ref{eq:performanceinequality}) while minimizing $\gamma$. We denote the resulting semi-definite programming (SDP) problem by P$1$. The result, $\gamma\!=\!3.47$,  helps in assessing the upper bound on $\|(G-G_r)\|$ and in choosing how many temporal and spatial state variables to truncate.  Then, we find $(0,28)$-ETP solutions to (\ref{eq:lyapinequ1}) which minimize $\sum_{t=0}^{27} (\sum_{k=1}^{5}\tr \, X^{(k)}(t)+ \sum_{(i,j)\in E} \tr \, X^{(ij)}(t))$,  and $(0,28)$-ETP solutions $Y^{(k)}(t)$, $Y^{(ij)}(t)$ to (\ref{eq:lyapinequ2}) which minimize a similar objective function. We denote the resulting SDP problems by P$2$ and P$3$, respectively. We determine the computational complexity of these problems by formulating the corresponding dual problems. The results are summarized in the table below. Let $N_I=12$ be the number of interconnections, and $n_u=2$ be the number of inputs to each subsystem.
\begin{align*}
\begin{array}{|c|c|c|}
\hline
                      & \mbox{P$1$}                                    & \mbox{P$2$ and P$3$}\\
\hline
\mbox{dimension of}   &   (h+q)(N(2 n_T+n_u)+2N_I n_S)                 & 2(h+q)(Nn_T+N_In_S)\\
\mbox{SDP variable}   &     =3976                                      & =3696 \\
\hline
\mbox{dimension of}   &      1                                         & 0\\
\mbox{linear variable}&                                                & \\
\hline
\mbox{number of}      &    (h+q)(2N +N_I)                              & (h+q)(2N +N_I) \\
\mbox{SDP blocks}     &         =616                                   & =616\\
\hline
\mbox{number of}      & \rule{0mm}{4.5mm}\frac{1}{2}(h+q)(Nn_T(n_T+1) +N_In_S(n_S+1))+1 & \frac{1}{2}(h+q)(Nn_T(n_T+1) +N_In_S(n_S+1))\\
\mbox{constraints}    &       =4957                                    & =4956\\
\hline
\end{array}
\end{align*}

We use Yalmip  to model these problems and SDPT3 to solve them; see \begin{NoHyper}{\cite{yalmip,sdpt3}}\end{NoHyper}. We carry out the computations in Matlab $7.10.0.499$ (The MathWorks Inc., Natick, Massachusetts, USA) on a Hewlett-Packard laptop with $2$ Intel Cores, $2.30$ GHz processors, and $4$ GB of RAM running Windows $7$. The most time consuming problem is P$1$. The corresponding elapsed time is about $30$ seconds (CPU time $24$ seconds).

We use Algorithm~\ref{algorithm} to construct a $(0,28)$-ETP balanced realization for $G$. To obtain useful error bounds for BT, we re-solve the Lyapunov inequalities for the balanced realization of $G$. Namely, we find $(0,28)$-ETP diagonal solutions $\Sigma^{(k)}(t)\succeq \epsilon I$ and $\Sigma^{(ij)}(t)\succeq \epsilon I$ that satisfy (\ref{eq:lyapinequ1}) and (\ref{eq:lyapinequ2}), and minimize the objective function
\[
\sum_{t=0}^{27} \left(\sum_{k=1}^5\left\|\vect\left(\Sigma^{(k)}(t)-\epsilon I\right)\right\|_1 +\sum_{(i,j)\in E}\left\|\vect\left(\Sigma^{(ij)}(t)-\epsilon I\right)\right\|_1\right)+a_1\times \epsilon,\]
where $\vect(M)$ is the vector formed by the diagonal entries of the matrix $M$ and $\|v\|_1$ is the $1$-norm of vector $v$.  $a_1=750$ is the weight given to $\epsilon$ in the cost function. This value of $a_1$ gives the best trade-off between the following two objectives. The first objective is to minimize the first term in the cost function, where the  $1$-norm is used as a heuristic for finding diagonal gramians with many  entries equal to $\epsilon$. We intend to truncate all the temporal and spatial state variables whose corresponding entries in the gramians are equal to $\epsilon$. The second objective is to minimize $\epsilon$ since, by Corollary~\ref{eventuallyperiodicerrorbound}, $\|(G-G_r)\|<2 \epsilon$. We get $\epsilon=0.034$, i.e., $\|(G-G_r)\|< 2\% \gamma$. Figure \ref{fig:diagonalentriesplot} shows the first and second diagonal entries of $\Sigma^{(21)}(t)$, for $0\le t <28$. The red dashed line corresponds to $\epsilon$. The dimension of $x^{(21)}$ varies between $0$ and $1$ in $G_r$ in contrast to $3$ in $G$, i.e., in $G_r$, the interconnection $(2,1)$ disappears at certain time-steps from the interconnection structure. As can be seen in Figure \ref{fig:truncatedvariables}, between $14$ and $18$ temporal variables and between $20$ and $25$ spatial variables are truncated at each time-step. Clearly, truncation need not be uniform in time even if the dimensions of the states in the full order system are constants. We simulate $G$ and $G_r$ for the same set of applied inputs and plot their responses in Figure \ref{fig:comparisonofGandGr}. The inputs vary randomly between $-10$ and $10$ for the first $100$ time-steps and then are set equal to zero. As predicted by the error bound, the responses of $G$ and $G_r$ are very close.

\begin{figure}[t]
\centering
\includegraphics[scale=0.4]{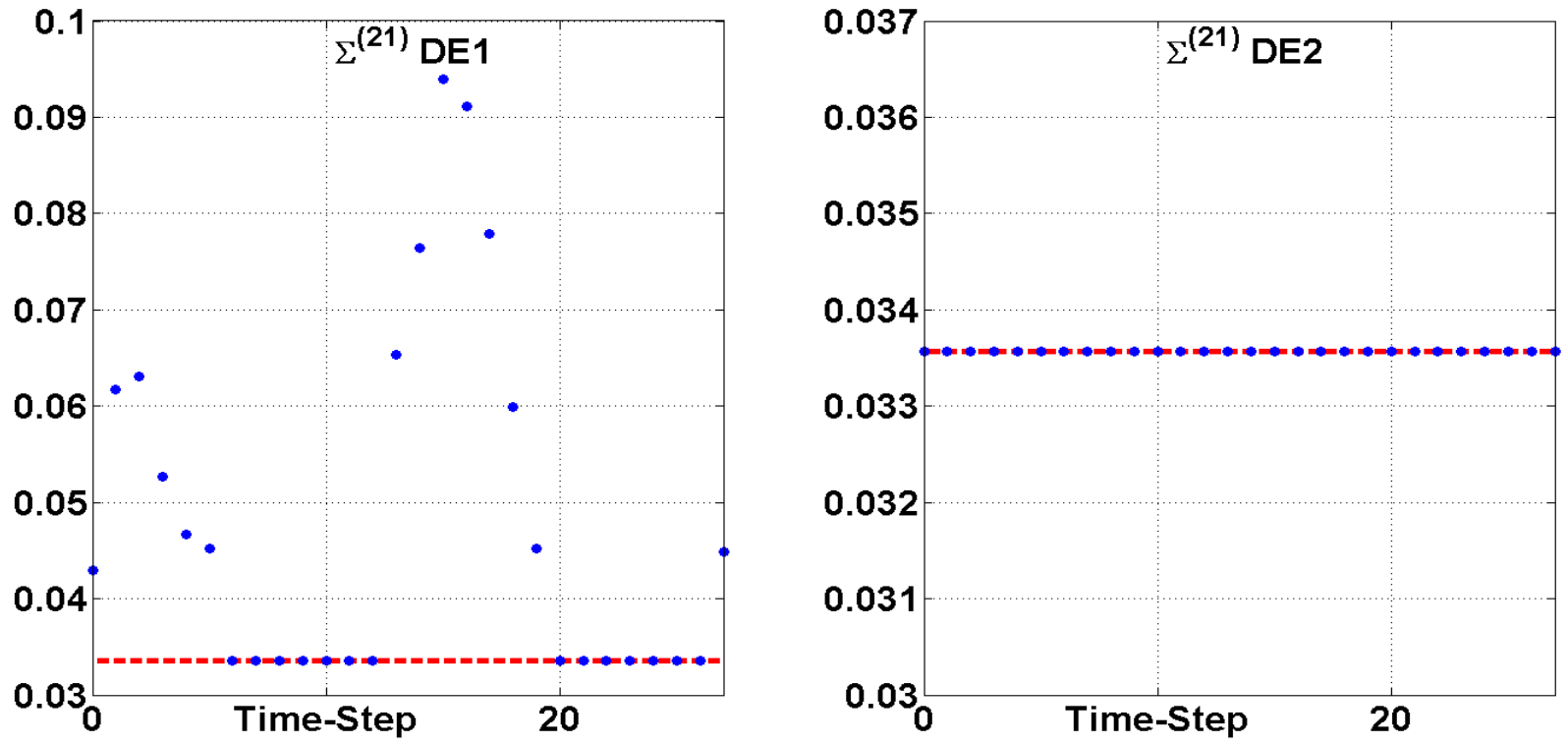}
\vskip -3mm
\caption{First and second diagonal entries (DE) of $\Sigma^{(21)}(t)$.}
\label{fig:diagonalentriesplot}
\end{figure}
\begin{figure}[t]
\centering
\includegraphics[scale=0.25]{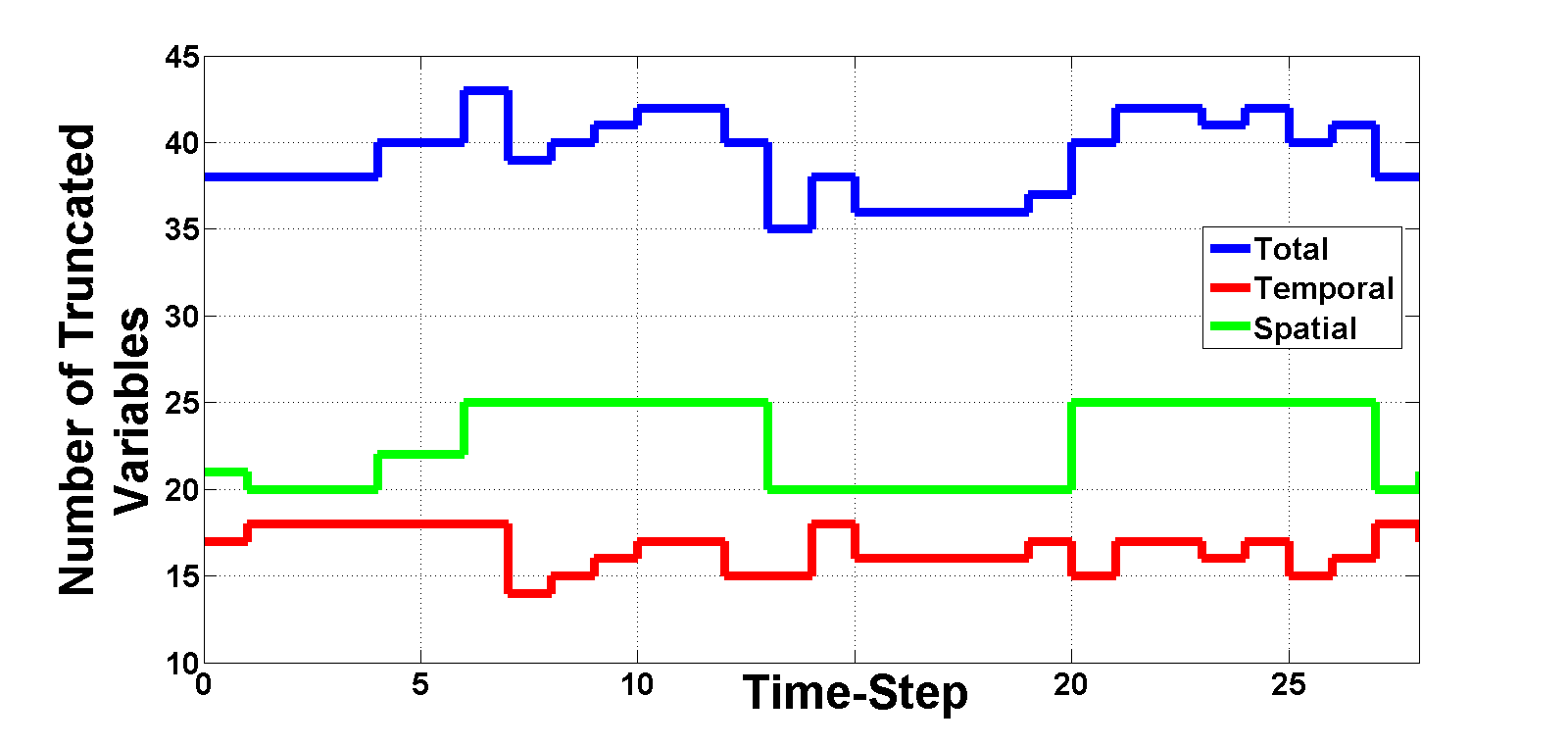}
\vskip -3mm
\caption{Number and type of truncated state variables as function of time.}
\label{fig:truncatedvariables}
\end{figure}
\begin{figure}[t]
\centering
\includegraphics[scale=0.4]{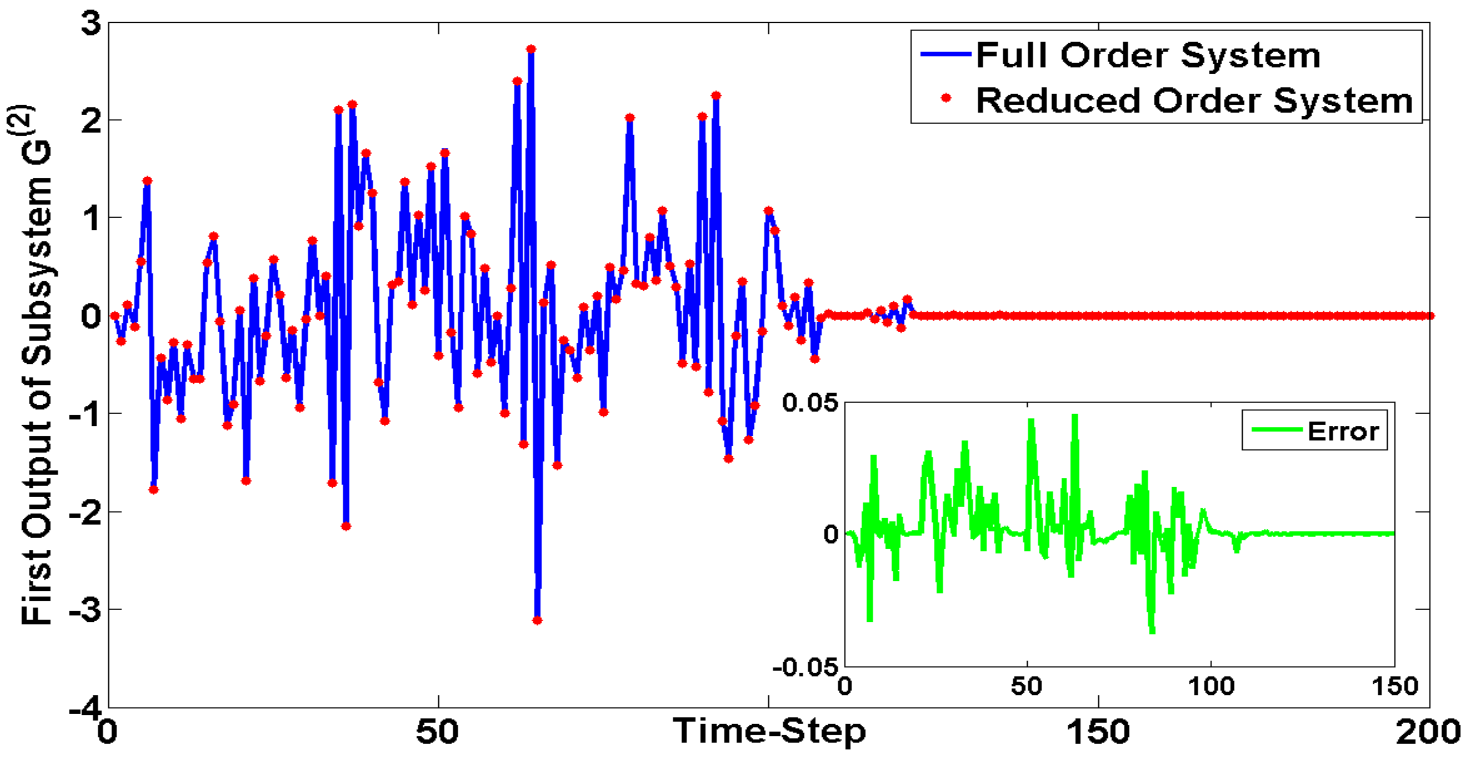}
\vskip -3mm
\caption{Responses of the full order and reduced order systems for the same set of applied inputs.}
\label{fig:comparisonofGandGr}
\end{figure}
\section{Conclusion}\label{conclusionsection}
This work applies BT for the model reduction of distributed LTV systems. The method provides a priori error bounds, preserves the interconnection structure, and allows for its simplification. While BT is only applicable to strongly stable systems, CFR extends the applicability of BT to strongly stabilizable and strongly detectable systems.

\section*{Acknowledgment}
This work is supported by the National Science Foundation under Grant CMMI-1333785.

\bibliographystyle{elsarticle-harv}
\bibliography{AbouJaoudeFarhood_arXiv}

\end{document}